\documentclass{amsart}
\usepackage{latexsym,amsxtra,amscd,ifthen}
\usepackage{amsfonts}
\usepackage{verbatim}
\usepackage{amsmath}
\usepackage{amsthm}
\usepackage{amssymb}

\numberwithin{equation}{section}

\theoremstyle{plain}
\newtheorem{theorem}{Theorem}[section]
\newtheorem{lemma}[theorem]{Lemma}
\newtheorem{proposition}[theorem]{Proposition}
\newtheorem{corollary}[theorem]{Corollary}

\theoremstyle{definition}
\newtheorem{definition}[theorem]{Definition}
\newtheorem{example}[theorem]{Example}

\newtheorem{remark}[theorem]{Remark}
\newtheorem*{remark*}{Remark}
\newtheorem{question}[theorem]{Question}

\makeatletter              
\let\c@equation\c@theorem  
\makeatother

\newcommand{\rank}{\operatorname{rank}}

\DeclareMathOperator{\Spec}{Spec}

\DeclareMathOperator{\Ext}{Ext}

\DeclareMathOperator{\gr}{gr}

\DeclareMathOperator{\GKdim}{GKdim}

\begin{document}

\title{Lower bounds of Growth of Hopf algebras}

\author{D.-G. Wang, J.J. Zhang and G. Zhuang}

\address{Wang: School of Mathematical Sciences,
Qufu Normal University, Qufu, Shandong 273165, P.R.China}

\email{dgwang@mail.qfnu.edu.cn, dingguo95@126.com}

\address{Zhang: Department of Mathematics, Box 354350,
University of Washington, Seattle, Washington 98195, USA}

\email{zhang@math.washington.edu}

\address{Zhuang: Department of Mathematics, Box 354350,
University of Washington, Seattle, Washington 98195, USA}

\email{gzhuang@math.washington.edu}

\begin{abstract}
Some lower bounds of GK-dimension of Hopf algebras are given.
\end{abstract}

\subjclass[2000]{Primary 16P90, 16W30; Secondary 16A24, 16A55}


\keywords{Hopf algebra, Gelfand-Kirillov dimension, skew primitive,
pointed}


\maketitle


\setcounter{section}{-1}
\section{Introduction}
\label{xxsec0}
A seminal result of Gromov states that a finitely generated group
has polynomial growth, or equivalently, the associated group algebra
has finite Gelfand-Kirillov dimension, if and only if it has a
nilpotent subgroup of finite index \cite{Gr}. Group algebras form a
special class of cocommutative Hopf algebras. It is natural to ask

\begin{question}
\label{xxque0.1}
What are necessary and sufficient conditions on a finitely
generated Hopf algebra $H$ such that its Gelfand-Kirillov dimension
is finite?
\end{question}

Let $k$ be a base field and everything be over $k$. Assume that, for
simplicity, $k$ is algebraically closed of characteristic zero. It
is clear that an affine (i.e., finitely generated) commutative Hopf
algebra has a finite GK-dimension (short for Gelfand-Kirillov
dimension) which equals its Krull dimension. If $H$ is
cocommutative, by a classification result \cite[Corollary 5.6.4 and
Theorem 5.6.5]{Mo}, it is isomorphic to a smash product
$U({\mathfrak g}) \# kG$ for some group $G$ and some Lie algebra
${\mathfrak g}$. Consequently,
\begin{equation}
\label{I0.1.1}\tag{I0.1.1}
\GKdim H=\GKdim kG +\dim {\mathfrak g},
\end{equation}
which solves Question \ref{xxque0.1} in terms of conditions on $G$
and ${\mathfrak g}$. Question \ref{xxque0.1} is also answered for
several classes of noncommutative and noncocommutative Hopf algebras,
including quantum groups $U_q({\mathfrak g})$ and ${\mathcal O}_q(G)$,
see \cite{BG2, GZ1}. The present paper attempts to study Question
\ref{xxque0.1} for a larger class of noncommutative and noncocommutative
Hopf algebras by providing three lower bounds of GK-dimension in terms
of certain invariants of skew primitive elements.

Let $H$ be a Hopf algebra over $k$. A nonzero element $y\in H$ is called
{\it $(1,g)$-primitive} (or generally skew primitive) if $\Delta(y)=
y\otimes 1+g\otimes y$ and such a $g$ is called the {\it weight} of $y$
and denoted by $\mu(y)$. Let $G(H)$ denote the group of group-like
elements in $H$ and let $C_0=kG(H)$. Here is the first lower bound theorem.

\begin{theorem}[First lower bound theorem]
\label{xxthm0.2}
Let $D\supseteq C_0$ be a Hopf subalgebra of $H$.
Let $\{y_i\}_{i=1}^w$ be a set of skew primitive elements such that
\begin{enumerate}
\item
$\{y_i\}_{i=1}^w$ is linearly independent in $H/D$.
\item
for all $i\leq j$, $y_i \mu(y_j)=\lambda_{ij}\mu(y_j)y_i$ for some
$\lambda_{ij}\in k^{\times}$,
\item
for each $i$, $\lambda_{ii}$ is either 1 or not a root of unity.
\end{enumerate}
Then $\GKdim H\geq \GKdim D+w$.
\end{theorem}

In general $\lambda_{ij}$ in condition (b) may not exist. If that is the
case, we have other ways of obtaining lower bounds.

Let $W$ denote the set of weights $\mu(y)$ for all skew primitive
elements $y\not\in C_0$ and let $W_{\sqrt{\;}}$ be the subset of $W$
consisting
of weights $\mu(y)$ for all $y$ such that $y^n$ is also a skew primitive
for some  $n>1$. (Note that in this paper the term ``skew primitive''
means ``$(1,g)$-primitive''). For any subset $\Phi\subset G(H)$,
the subgroup of $G(H)$ generated by $\Phi$ is denoted by $\langle
\Phi \rangle$. Here is the second lower bound theorem.

\begin{theorem}[Second lower bound theorem]
\label{xxthm0.3}
Suppose $\langle W\setminus W_{\sqrt{\;}} \rangle$ is abelian. Then
\begin{equation}
\label{I0.3.1}\tag{I0.3.1}
\GKdim H\geq \GKdim C_0+\#(W\setminus W_{\sqrt{\;}}).
\end{equation}
\end{theorem}

There are examples such that $W=W_{\sqrt{\;}}$ and $\GKdim H=
\GKdim C_0$, but $\#(W_{\sqrt{\;}})$ is arbitrarily large
[Example \ref{xxex2.7}]. Therefore $W_{\sqrt{\;}}$ has to be
removed from $W$ when we estimate the GK-dimension of $H$.

Let $y$ be a skew primitive element not in $C_0$. If
\begin{equation}
\label{I0.3.2}\tag{I0.3.2}
\mu(y)^{-1}y\mu(y)-c y\in C_0
\end{equation}
for some $c\in k^\times$, then $c$
is called the {\it commutator} of $y$ (with its weight) and denoted
by $\gamma(y)$. By Lemma \ref{xxlem1.6}, \eqref{I0.3.2} is
equivalent to
\begin{equation}
\label{I0.3.3}\tag{I0.3.3}\mu(y)^{-1}y\mu(y)-c y=\tau(\mu(y)-1)
\end{equation}
for some $\tau\in k$.
Define $\Gamma$ to be the set of $\gamma(y)$ for
all skew primitive elements $y\not\in C_0$ such that $\gamma(y)$
exists and let $\Gamma_{\sqrt{\;}}$ be the subset of $\Gamma$
consisting of those $\gamma(y)$ which are roots of unity but not 1.
If $\gamma(y)$ exists, the pair $(\mu(y),\gamma(y))$
is denoted by $\omega(y)$ and is called the {\it weight 
commutator} of $y$. When \eqref{I0.3.3} holds and if $c\neq 1$,
$y$ can be replaced by $z:=y+(c-1)^{-1}\tau(\mu(y)-1)$, which is a
skew primitive element with $\omega(z)=\omega(y)$ and satisfies
the equation $\mu(z)^{-1} z \mu(z)-\gamma(z) z=0$.

Define $\Omega$ to be the set of
$\omega(y)$ for all skew primitive elements $y\not\in C_0$ such
that $\omega(y)$ exists and
let $\Omega_{\sqrt{\;}}$ be the subset of $\Omega$ consisting of
those $\omega(y)$ in which $\gamma(y)$ is a root of unity but not 1.
Theorem \ref{xxthm0.3} can be improved a little under the same
hypothesis:
$$\GKdim H\geq \GKdim C_0+\#(\Omega\setminus \Omega_{\sqrt{\;}}).$$

Let $y$ be a skew primitive element not in $C_0$ with $g=\mu(y)$. 
Let $T_{g^{-1}}$
be the inverse conjugation by $g$, namely, $T_{g^{-1}}: a\to g^{-1}a g$.
A scalar $c$ is called a {\it commutator of $y$ of level $n$} if $n$ is
the least nonnegative integer such that
\begin{equation}
\label{I0.3.4}\tag{I0.3.4}
(T_{g^{-1}}-c Id_H)^n(y)\in C_0.
\end{equation}
In this case we also write $\gamma(y)=c$. Let $Z$ denote the space
spanned by the identity element $1$ and all skew primitive elements
of $H$ and let $Y_{\sqrt{\;}}$
denote the subspace of $Z$ spanned by those $y$ with commutator
of finite level and with $\gamma(y)$ being a root of unity but
not 1. Here is the third lower bound theorem. Let $W_{\times}$ 
be the subset of $W$ consisting of weights $\mu(y)$ such that 
the commutator of $y$ (as defined in \eqref{I0.3.4}) exists and is either
1 or not a root of unity. Note that $W\setminus W_{\sqrt{\;}}
\subseteq W_{\times}$ and these are often equal [Remark 
\ref{xxrem3.9}].

\begin{theorem}[Third lower bound theorem]
\label{xxthm0.4}
Suppose $\langle W_{\times}\rangle$ is abelian. Then
\begin{equation}
\label{I0.4.1}\tag{I0.4.1}
\GKdim H\geq \GKdim C_0+\dim Z/(C_0+Y_{\sqrt{\;}}).
\end{equation}
\end{theorem}

When $H$ is cocommutative, equality holds in Theorem \ref{xxthm0.4},
see \eqref{I0.1.1}. There are examples such that $Z=Y_{\sqrt{\;}}+C_0$
and $\GKdim H= \GKdim C_0$, but $\dim Y_{\sqrt{\;}}$ is arbitrarily
large [Examples \ref{xxex2.7} and \ref{xxex3.13}]. Therefore it is
sensible to consider the quotient space $Z/(C_0+Y_{\sqrt{\;}})$ in the
above theorem. This is analogous to removing $W_{\sqrt{\;}}$ in Theorem
\ref{xxthm0.3}.

If $\langle W_{\times}\rangle$ is abelian, Theorem \ref{xxthm0.4} is 
a generalization of Theorem \ref{xxthm0.3} [Lemma \ref{xxlem3.12}]. 
After some analysis, Theorem \ref{xxthm0.2} (when $D=C_0$) can be viewed 
as a consequence of Theorem \ref{xxthm0.4}.
These lower bounds provide some evidence that the GK-dimension of
$H$ is related to some combinatorial data coming from the skew
primitive elements when $H$ is pointed.

The proof of these lower bounds is based on a version of the
Poincar{\'e}-Birkhoff-Witt (PBW) theorem [Theorem \ref{xxthm1.5}(b)]
which states that under some hypotheses the set of monomials generated
by skew primitive elements is linearly independent (over the Hopf
subalgebra $C_0$). Restricted to
the universal enveloping algebra of a finite dimensional Lie algebra,
Theorem \ref{xxthm1.5} implies the original PBW theorem. Theorem
\ref{xxthm1.5} is in a similar spirit to Kharchenko's quantum
analog of the PBW theorem \cite{Kh}. One of the hypotheses in
Theorem \ref{xxthm1.5} is (I1.2.3) which assume essentially
that the action of the group generated by weights on the space
generated by skew primitive elements is locally finite. When $\GKdim H$
is finite, this is a reasonable hypothesis indicated by a result of
the third-named author \cite[Theorem 1.2]{Zhu}
(see also Lemma \ref{xxlem2.5}).

In general we are far from answering Question \ref{xxque0.1}. There
are a lot of unsolved questions concerning the growth of Hopf algebras.
The hypotheses in Theorems \ref{xxthm0.3} and \ref{xxthm0.4} could be 
superfluous, but we don't know how to remove them at this moment.  When 
$\langle W\rangle$ is non-abelian, a possible better lower bound could 
be obtained by replacing $\#(W\setminus W_{\sqrt{\;}})$ in Theorem 
\ref{xxthm0.3} by $\GKdim k\langle W\rangle$, see Lemma \ref{xxlem2.6}(b) 
for details. It is expected that these lower bounds can (or should) be 
improved and that possible upper bounds should be found once finer 
invariants are introduced. The ultimate goal is to find a formula for 
the GK-dimension of a Hopf algebra which is analogous to Bass' theorem 
\cite[Theorem 11.14]{KL} in the group algebra case, and then eventually 
to solve Question \ref{xxque0.1}.

There are further connections between the growth of Hopf algebras
and $W$ and other invariants defined by skew primitive elements.
Let $\rank$ denote the torsionfree rank of an abelian group.

\begin{proposition}
\label{xxprop0.5}
Suppose $\langle W \rangle$ is abelian
and torsionfree. If  $\rank \langle \Gamma \rangle >\rank \langle
W\rangle=1$, then $H$ has exponential growth.
\end{proposition}

Note that $\rank \langle \Gamma \setminus \Gamma_{\sqrt{\;}}\rangle
=\rank \langle \Gamma \rangle$ since elements in $\Gamma_{\sqrt{\;}}$
have finite order. The rank of $\langle W \rangle$ and
$\langle \Gamma \rangle$ should be related when $\GKdim H$ is finite.

\begin{question}
\label{xxque0.6}
Suppose $\rank \langle \Gamma \rangle >\rank \langle W \rangle$.
Does then $H$ have exponential growth?
\end{question}

Quite a few families of Hopf algebras of finite GK-dimension have been
analyzed extensively by several authors \cite{AA, AS1, AS2, Br1, Br2,
BG1, BG2, BZ, GZ1, GZ2, LWZ, WuZ1, WuZ2, Zhu} during the last few years.
But the classification of such Hopf algebras is far from complete.
These lower bounds are useful for studying pointed Hopf algebras
of low GK-dimension. For example, if $\GKdim H=2$, then there are
only three possibilities for $\GKdim C_0$, $\#(W\setminus W_{\sqrt{\;}})$,
$\#(\Omega\setminus \Omega_{\sqrt{\;}})$ and $\dim Z/(C_0+Y_{\sqrt{\;}})$.
This is one of the initial steps in our ongoing project of
classifying pointed Hopf algebra domains of GK-dimension two and three.

Definitions and basic properties of GK-dimension can be found in the
first three chapters of \cite{KL}. Our reference book for Hopf
algebras is \cite{Mo}.

\section{First Lower Bound Theorem}
\label{xxsec1}

In this section we prove Theorem \ref{xxthm0.2}. We need some
lemmas.

\begin{lemma}
\label{xxlem1.1}
Let $D$ be a Hopf subalgebra of $H$ and $0\neq F\in H$. Suppose that
\begin{enumerate}
\item
$L$ is a subcoalgebra of $H$ containing $D$,
\item
$L$ is a left $D$-module via the multiplication, and
\item
there are nonzero-divisors (regular elements) $h,g\in L$ such that
$\Delta(F)-F\otimes h-g\otimes F\in L\otimes L$.
\end{enumerate}
Define
$V=\{a\in D\mid a F \in L\}.$
Then $V$ is either 0 or $D$.
\end{lemma}

\begin{proof}
Suppose $V$ is nonzero and let $a$ be a nonzero element in $V$. Let
$C$ be the subcoalgebra of $D$ generated by $a$. There is a
$k$-linear basis
$$\{a_1,\cdots, a_v,a_{v+1},\cdots, a_w\}$$
of $C$ such that $C\cap V$ is spanned by $\{a_1,\cdots,a_v\}$.
This means that $a_i F\in L$ for all $i\leq v$ and that any
nontrivial linear combination of $\{a_{v+1}F,\cdots,a_{w}F\}$ is
not in $L$. Write
$\Delta(a)=\sum_{1\leq i,j\leq w} \xi_{ij} a_i\otimes a_j$
for some $\xi_{ij}\in k$.

For simplicity, we use the symbol $ldt_1$ for any element in $L$ and
use $ldt_2$ for any element in $L\otimes L$. By the definition of
$V$, we have $aF+ldt_1=0$ for some $ldt_1\in L$ and whence
$$
\begin{aligned}
0&=\Delta(aF+ldt_1)=\Delta(a) \Delta(F)+\Delta(ldt_1)\\
&=(\sum_{i,j}\xi_{ij} a_i\otimes a_j)
(F\otimes h+g\otimes F+ldt_2)+ldt_2\\
&=(\sum_{i,j}\xi_{ij} a_i\otimes a_j)
(F\otimes h)+(\sum_{i,j}\xi_{ij} a_i\otimes a_j) (g\otimes F)+ldt_2\\
&=(\sum_{\substack{i>v\\ \mathrm{all}\,j}}\xi_{ij} a_i\otimes a_j)
(F\otimes h)+(\sum_{\substack{j>v\\ \mathrm{all}\,i}}\xi_{ij} a_i\otimes a_j)
(g\otimes F)+ldt_2
\end{aligned}
$$
where the last equation uses the fact $a_i F\in L$ for all $i\leq v$.
The above equation implies that
$$(\sum_{\substack{j>v\\ \mathrm{all}\,i}}\xi_{ij} a_i\otimes a_j)(g\otimes F)
=-(\sum_{\substack{i>v\\ \mathrm{all}\,j}}\xi_{ij} a_i\otimes a_j)
(F\otimes h)+ldt_2
\in H\otimes L$$
or equivalently
$\sum_{i=1}^w (a_i g)\otimes (\sum_{j>v} \xi_{ij} a_j F)\in H\otimes L$.
Since $\{a_i g\}_{i=1}^w$ is linearly independent, we have
$\sum_{j>v} \xi_{ij} a_j F\in L$ for all $i$. By the definition of
$\{a_{v+1},\cdots,a_w\}$, we obtain that $\xi_{ij}=0$ for all $j>v$.
Similarly, $\xi_{ij}=0$ for all $i>v$. Thus $\Delta(a)\in V\otimes V$
and hence $V$ is a subcoalgebra of $D$. Since $V$ is a subcoalgebra,
there is an element $v\in V$ such that $\epsilon(v)=1$. Then
$$1 F=\epsilon(v) F=\sum S(v_1) v_2 F\in L$$
since $v_2\in V$ and $S(v_1)\in D$. This shows that $1\in V$. Since
$V$ is a left ideal of $D$, $V=D$.
\end{proof}

\begin{remark}
\label{xxrem1.2} Recall that $({\mathbb N}^v,+)$, for every $v\geq 1$, 
is a linearly ordered semigroup with respect to the following ordering. 
Define $(c_1,\cdots,c_v)<
(d_1,\cdots, d_v)$ if either $\sum_{i=1}^v c_i<\sum_{i=1}^v d_i$ 
or $\sum_{i=1}^v c_i=\sum_{i=1}^v d_i$ and there is a $p<v$ such 
that $c_i=d_i$ for all $i\leq p$ and $c_{p+1}<d_{p+1}$. 

Let $T:=k\langle \{x_j\}_{j\in J}\rangle$ be the free algebra generated
by $\{x_j\}_{j\in J}$. Given any family $(f_j)_{j\in J}$ where $f_j\in 
{\mathbb N}^v$, 
we can define an ${\mathbb N}^v$-graded structure 
on $T$ by setting $\deg x_j=f_j$ for all $j\in J$. Then $T=\bigoplus_{w\in 
{\mathbb N}^v} T_w$. Since 
${\mathbb N}^v$ is linearly ordered, $T$ has a canonical
${\mathbb N}^v$-filtration defined by $F_{w}(T)=\sum_{w'\leq w} T_{w'}$.
Let $B$ be any factor ring of $T$. The ${\mathbb N}^v$-filtration 
on $T$ induces a unique ${\mathbb N}^v$-filtration on $B$, denoted by
$\{F_w(B)\mid w\in {\mathbb N}^v\}$. We say that an element $x\in B$ has 
{\it filtered multi-degree} 
$$\deg x:=w=(d_1, \cdots, d_v)$$
and {\it filtered total-degree} $d=\sum d_i$ if $x\in F_w(B)\setminus 
\sum_{w'<w} F_{w'}(B)$. Note that the filtered total-degree induces an 
$\mathbb{N}$-filtration on $B$.

In applications, we usually start with an algebra $A$ generated by 
$\{y_1,\cdots,y_v\}$ and $G=\{g_i\}_{i\in I}$ for some index set $I$. 
By the discussion in the previous paragraph, we 
can define two filtrations (and corresponding filtered degrees) on $A$ 
such that if $f=y_{i_1}y_{i_2} \cdots y_{i_s}\in A$, then
the filtered total $y$-degree of $f$ is at most $s$ and the filtered 
multi-$y$-degree of $f$ is 
at most $(n_1,\cdots, n_v)$ where $n_i$ is the number of $y_i$ appearing 
in $F$, and if $g\in G$, the filtered total-$y$-degree and the filtered 
multi-$y$-degree of $g$ are both 0.

These two filtrations can be extended to the tensor product 
$A\otimes A$, namely, $F_{w}(A\otimes A):=\sum_{w'+w''\leq w} 
F_{w'}(A)\otimes F_{w''}(A)$ for all $w\in {\mathbb N}^v$ (or 
$w\in {\mathbb N}$). For simplicity, the words  ``filtered" and 
``filtration" might be omitted below.
\end{remark}

Assume that $S:=\{y_i\}_{i\in I}$ is a set of skew primitive elements of 
$H$ where $I$ is either ${\mathbb N}$ or $\{1,\cdots,v\}$ for some 
positive integer $v$. Suppose that $D$ is a Hopf subalgebra of $H$ and that
\begin{enumerate}
\item[(I1.2.1)]
$g_i:=\mu(y_i)\in D$ for all $i\in I$,
\item[(I1.2.2)]
$S$ is linearly independent in the space $H/D$,
\item[(I1.2.3)]
for each pair $i\leq j$, $y_ig_j=\lambda_{ij} g_j y_i+b_{ij}$ for some
$\lambda_{ij}\in k^\times$ and $b_{ij}\in D$, and there is a subalgebra
$A\subset D$ containing all $b_{ij}$ such that
$y_i A\subset A y_i+A $ and $g_i A\subset A g_i+A$ for all $i$.
\end{enumerate}
In most of the applications $D$ is the coradical $C_0$ of $H$ and the
commutators of the $y_i$ exist. When $b_{ij}=0$ for all $i\leq j$, we may
take $A=k$ and then {\rm{(I1.2.3)}} is automatic.
For every positive integer $d$, define
$$S^d:=\{y_1^{d_1}\cdots y_n^{d_n} \cdots \mid \sum_s d_s=d\}.$$

The following lemma is known
and easy to check by a direct computation.

\begin{lemma}
\label{xxlem1.3}
Suppose {\rm{(I1.2.1)-(I1.2.3)}} hold.
\begin{enumerate}
\item
For every $n$,
$$\Delta(y_i^n)=\sum_{s=0}^n {n \choose s}_{\lambda_{ii}} g_i^{s} y_i^{n-s}
\otimes y_i^s+\sum_{s+s'<n} a_{ss'} y_i^s\otimes y_i^{s'}$$
for some $a_{ss'}\in \sum_{t\geq 0} A g_i^t$. If $b_{ii}=0$, then
$a_{ss'}=0$ for all $s,s'$.
\item
Let $\{y_1,y_2,\cdots,y_z\}$ be a finite subset of $S$. Then, for
$n_1,\cdots, n_z\geq 0$,
$$\begin{aligned}
\Delta(y_1^{n_1}\cdots y_{z}^{n_z})=&
\sum_{s_1,\cdots, s_z} (\prod_{t=1}^z {n_t \choose s_t}_{\lambda_{tt}}) 
c_{(s_t)}
g_{1}^{s_1}\cdots g_z^{s_z} y_1^{n_1-s_1}\cdots y_z^{n_z-s_z}
\otimes y_1^{s_1}\cdots y_z^{s_z}\\
&+ldt_2
\end{aligned}
$$
where $c_{(s_t)}= \prod_{i<j} \lambda_{ij}^{s_j(n_i-s_i)} \in k^{\times }$.
Here $ldt_2$ is a linear combination of elements of the form $f
y_1^{a_1}\cdots y_z^{a_z} \otimes y_1^{b_1}\cdots y_z^{b_z}$ with 
$\sum_{i}(a_i+b_i)<\sum_{i}n_i$ where
$f\in \sum_{t_1,\cdots, t_z\geq 0} A g_1^{t_1}\cdots g_z^{t_z}$.
If $b_{ij}=0$ for all $i\leq j$, then $ldt_2=0$.
\end{enumerate}
\end{lemma}

For $\alpha=(n_1, \cdots, n_z, 0, \cdots )$, define
\begin{equation}
\label{I1.3.1}\tag{I1.3.1}
L_{\alpha}=\sum_{G} DG
\end{equation}
where $G$ runs through elements $y_1^{m_1}\cdots y_{w}^{m_w}$
such that $(m_1,\cdots,m_w,0,\cdots)<\alpha$. 

\begin{lemma}
\label{xxlem1.4}
Retain the notation as above and suppose {\rm{(I1.2.1)-(I1.2.3)}}
hold. Let $\alpha=(n_1, \cdots, n_z, 0, \cdots )$ and $F=y_1^{n_1}\cdots 
y_z^{n_z}$. Define
$$V=\{a\in D\mid a F\in L_{\alpha}\}.$$
Then $V$ is either 0 or $D$.
\end{lemma}

\begin{proof} Let $L$ denote $L_{\alpha}$ in the proof.
First we claim that $\Delta(L)\subset L\otimes L$. It suffices to
show that $\Delta(G)\in L\otimes L$ for all $G=y_1^{m_1}\cdots y_{w}^{m_w}$
with $(m_1,\cdots,m_w,0,\cdots)<\alpha$. By Lemma \ref{xxlem1.3},
$$
\Delta(G)=G\otimes 1+g_{1}^{m_1}\cdots g_{w}^{m_w}\otimes G+ldt_2 
\in L\otimes L.
$$
Thus we proved our claim. It is easy to see that the hypotheses in
Lemma \ref{xxlem1.1}(a,b) hold.
For the  hypothesis in Lemma \ref{xxlem1.1}(c), we note that
$$\Delta(F)=F\otimes 1+g_1^{n_1}\cdots g_z^{n_z}\otimes F+ldt'_2$$
by Lemma \ref{xxlem1.3}(b), where $ldt'_2\in L\otimes L$. The
assertion follows from Lemma \ref{xxlem1.1}.
\end{proof}

Here is the main result of this section. Recall that $g_i=\mu(y_i)$
for all $i$.

\begin{theorem}
\label{xxthm1.5}
Assume that {\rm{(I1.2.1)-(I1.2.3)}} hold. Let $\lambda_i$ denote
$\lambda_{ii}$ for all $i$.
\begin{enumerate}
\item
Suppose the elements in $\bigcup_{j\geq 0} S^j$ are linearly dependent over
$D$ (on the left or on the right). Then there is some $z\in \mathbb{N}$ 
such that
\begin{enumerate}
\item[(i)]
$\lambda_{z}$ is a primitive $p_z$-th  root of unity for some $p_z>1$,
\item[(ii)]
there are $a_i,b_j\in k$, $p_j\in {\mathbb N}$ such that
$y_z^{p_z}+\sum_{i} a_i y_i+ \sum_{j\neq z} b_j y_j^{p_j} \in D$,
\item[(iii)]
$g_i=g_z^{p_z}$ whenever $a_i\neq 0$ in part (ii), and
\item[(iv)]
$g_j^{p_j}=g_z^{p_z}$ and $\lambda_j$ is a primitive $p_j$-th  root of
unity whenever $b_j\neq 0$ in part (ii).
\end{enumerate}
\item
Suppose $\lambda_{i}$ is either 1 or not a root of unity for every $i$.
Then the elements in $\bigcup_{j\geq 0} S^j$ are linearly independent
over $D$ (on the left and on the right). As a consequence,
$$\GKdim H\geq \GKdim D+\#(S).$$
\end{enumerate}
\end{theorem}

\begin{proof} (a)
Suppose that $\bigcup_{j\geq 0}
S^j$ is linearly dependent over $D$ on the left. Then there is an 
$F=y_1^{n_1}\cdots y_z^{n_z}\in S^d$ for some $d\ge 0$ such that
\begin{equation}
\label{I1.5.1}\tag{I1.5.1}
a F\in L_{\alpha}, \quad {\text{for some}}\quad 0\neq a\in D,
\end{equation}
where $\alpha=(n_1, \cdots, n_z, 0, \cdots )$. The definition of 
$L_{\alpha}$ is given in \eqref{I1.3.1}. Choose $F$ among all $(a,F)$ 
satisfying \eqref{I1.5.1} so that $\alpha$ is minimal with respect 
to the linear order $<$ defined in the beginning of Remark \ref{xxrem1.2}.
For simplicity let $L=L_{\alpha}$ for the rest of the proof. Let 
$V=\{b\in D\mid bF\in L\}$. Then $0\neq a\in V$. By Lemma \ref{xxlem1.4}, 
$1\in V$, or equivalently, $F\in L$. So we can write $F=ldt_1$ where $ldt_1$
denotes any element in $L$. By the minimality of $\alpha$, $L$
is a free left $D$-module with a basis $\{y_1^{m_1}\cdots y_{w}^{m_w}
\mid (m_1,\cdots,m_w,0,\cdots)<\alpha\}$. 
Note that $L\otimes L$ is a free $D\otimes D$-module with a basis
$$\{y_1^{m_1}\cdots y_{w}^{m_w}\otimes y_1^{l_1}\cdots y_{w'}^{l_{w'}}
\mid (m_1,\cdots,m_w,0,\cdots), (l_1,\cdots, l_{w'}, 0,\cdots)<\alpha\}.$$
We define a multi-degree on $L$ such that, for any nonzero $a\in D$, 
$\deg(a)=0$ and $\deg(a y_1^{m_1}\cdots y_{w}^{m_w})= (m_1, \cdots, 
m_w, 0, \cdots)$ whenever $(m_1, \cdots, m_w, 0, \cdots)<\alpha$.
Notice that under this definition $L$ is a graded $D$-module (but not an 
algebra), which can be viewed as a filtered $D$-module obviously. 
Extend this multi-grading naturally to $L\otimes L$ by adding 
the multi-degrees of the tensor components.

Recall that $F=y_1^{n_1}\cdots y_z^{n_z}$. We may
assume $n_z>0$ (if not, delete $y_z$ in the expression of $F$). Following
the last paragraph, there is an $ldt_1\in L$ such that $F=-ldt_1$, or
equivalently, $y_1^{n_1}\cdots y_z^{n_z}+ldt_1=0$. By the choice
of $F$, any element in $L$ has multi-degree less than
$\alpha$. Let $ldt_2$ denote any
element in $L\otimes L$ and let $lmt_2$ denote any element in
$L\otimes L$ with multi-degree less than $\alpha$.
Since the multi-degree of $ldt_1$ is less than $\alpha$,
$\Delta(ldt_1)$ is an $lmt_2$ by Lemma \ref{xxlem1.3}. Then, by 
Lemma \ref{xxlem1.3} again, we have
\begin{align}
\label{I1.5.2}\tag{I1.5.2}
0&=\Delta(F+ldt_1)=\Delta(F)+lmt_2\\
\notag
&= \sum_{s_1,\cdots, s_z}
(\prod_{t=1}^z {n_t \choose s_t}_{\lambda_t}) c_{(s_t)}
g_{1}^{s_1}\cdots g_z^{s_z} y_1^{n_1-s_1}\cdots y_z^{n_z-s_z}
\otimes y_1^{s_1}\cdots y_z^{s_z}+lmt_2\\
\notag
&= \sum_{(s_t)\neq (0),(n_t)}
(\prod_{t=1}^z {n_t \choose s_t}_{\lambda_t}) c_{(s_t)}
g_{1}^{s_1}\cdots g_z^{s_z} y_1^{n_1-s_1}\cdots y_z^{n_z-s_z}
\otimes y_1^{s_1}\cdots y_z^{s_z}\\
\notag
&\qquad + F\otimes 1+ g_{1}^{n_1}\cdots g_z^{n_z}\otimes F+lmt_2\\
\notag
&= \sum_{(s_t)\neq (0), (n_t)}
(\prod_{t=1}^z {n_t \choose s_t}_{\lambda_t}) c_{(s_t)}
g_{1}^{s_1}\cdots g_z^{s_z} y_1^{n_1-s_1}\cdots y_z^{n_z-s_z}
\otimes y_1^{s_1}\cdots y_z^{s_z}+lmt_2
\end{align}
where $lmt_2$ represents an element in $L\otimes L$
with multi-degree less than $\alpha$.
The multi-degree of $g_1^{n_1}\cdots g_z^{n_z} y_1^{n_1-s_1}\cdots
y_z^{n_z-s_z} \otimes y_1^{s_1}\cdots y_z^{s_z}$ equals
$\alpha$ for any $(s_t)\neq (0), (n_t)$. Using the fact that $L$ is
a free $D$-module with basis  $\{y_1^{m_1}\cdots y_{w}^{m_w}
\mid (m_1,\cdots,m_w,0,\cdots)<\alpha\}$, we obtain that
$(\prod_{t=1}^z {n_t \choose s_t}_{\lambda_t}) c_{(s_t)}=0$ or
$\prod_{t=1}^z {n_t \choose s_t}_{\lambda_t}=0$
for all $(s_t)\neq (0), (n_t)$. If $n_j>0$ for some $1\leq j<z$,
we take $(s_t)=(0,0,\cdots, 0,n_z)$, then
$\prod_{t=1}^z {n_t \choose s_t}_{\lambda_t}=1$, a contradiction.
Therefore $n_j=0$ for all $j<z$ which means that $F=y_z^{n_z}$.

If $n_z=1$, we have $y_z=\sum_{i<z} b_i y_i +c$ for $c,b_i\in D$.
Hence $\sum_{i<z} b_i y_i +c$ is $(1,g_z)$-primitive. Then
applying $\Delta$ we obtain that
$$\begin{aligned}
\Delta(b_i)&=b_i\otimes 1, \\
\Delta(b_i)(g_i\otimes 1)&=g_z\otimes b_i,\\
\Delta(c)&=c\otimes 1+g_z\otimes c.
\end{aligned}
$$
These imply that $b_i\in k$ and $g_i=g_z$ when $b_i\neq 0$.
This contradicts (I1.2.2). Therefore  $n_z>1$.

By the last two paragraphs, $n_z>1$ and $n_i=0$ for all $i<z$ and
${n_z \choose s_z}_{\lambda_z}=0$ for all $1\leq s_z\leq n_z-1$.
This can only happen when $\lambda_z$ is a primitive $n_z$-th root of unity
\cite[Lemma 7.5]{GZ2}.

Next let us re-name $n_z$ by $p_z$ and write $F=y_z^{p_z}$. Then 
$y_z^{p_z}+\sum_{i} b_i G_i+c_0=0$
where $b_i, c_0\in D$ and the $G_i$ are monomials 
with multi-$y$-degree less than $(0,\cdots, 0,p_z,0,\cdots)$ (where $p_z$ 
is in the $z$-th position). Repeating a computation similar
to \eqref{I1.5.2} (and the induction on the multi-$y$-degree of $G_i$) one can
show that each $G_i$ (when $b_i\neq 0$) is of the form $y_i^{n_i}$
and each $y_i^{n_i}$ is a skew primitive. If $n_i>1$, then
$\lambda_i$ is a primitive $n_i$-th  root of unity. In summary,
when $\lambda_i$ is not a root of unity, then $n_i=1$ and when
$\lambda_i$ is a primitive $p_i$-th root of unity, then $n_i$ is either 1
or $p_i$. So we have
$$-y_z^{p_z}=\sum_{i} a_i y_i+\sum_{j\neq z} b_j y_j^{p_j}+c$$
where $0\neq a_i, b_j\in D$ and $c\in D$. Thus $\sum_{i} a_i y_i+
\sum_{j} b_j y_j^{p_j}+c$ is $(1,g_z^{p_z})$-primitive. Since $L$
is a free left $D$-module, each of the nonzero $a_i y_i$, $b_j y_j^{p_j}$
and $c$ is $(1,g_z^{p_z})$-primitive. The coproduct computation shows
that $a_i, b_j\in k$ and $g_i=g_z^{p_z}$ and $g_j^{p_j}=g_z^{p_z}$.

(b) The first assertion is an immediate consequence of part (a).
To prove the second assertion, we take $W$ to be a finite dimensional
subspace of $D$ and let $S$ be a finite set $\{y_1,\cdots,y_z\}$. 
For a subspace $V\subset H$, let $V^n$ be the linear span of all 
elements $v_1\cdots v_n$ for $v_i\in V$. By the first assertion,
$$\begin{aligned}
\dim \; (W+k1+\sum_{i=1}^z ky_i)^{2n}&\geq 
\dim W^{n}(k1+\sum_{i=1}^z k y_i)^n\\
&\geq (\dim W^n) \# (\bigcup_{d=0}^n S^d)
\geq (\dim W^n) c n^z
\end{aligned}
$$
for some positive constant $c$. This implies that
$\GKdim H\geq \GKdim D+\# S$.
If $S$ is infinite, let $S'$ be any finite subset of $S$. Then
the above argument shows that $\GKdim H\geq \GKdim D+\# S'$ for
any $S'$. Thus  $\GKdim H=\infty=\GKdim D+\# S$.
\end{proof}

\begin{proof}[Proof of Theorem \ref{xxthm0.2}]
Let $S=\{y_1,\cdots,y_w\}$. Then (I1.2.1)-(I1.2.3)
follow easily from (a) and (b). The hypothesis in Theorem
\ref{xxthm1.5}(b) is the same as that in Theorem \ref{xxthm0.2}(c). 
Therefore the assertion follows from Theorem \ref{xxthm1.5}(b).
\end{proof}

The following easy lemma will be used implicitly later.

\begin{lemma}
\label{xxlem1.6}
Let $P$ be the set of all skew primitive elements in a Hopf algebra $H$
with weight $\mu$. Then $P$ is a $k$-subspace of $H$ and $P\cap C_0=
k(\mu-1)$.
\end{lemma}

\begin{proof} It is clear that $P$ is a $k$-subspace of $H$. For any
element $y\in P\cap C_0$, write $y=\sum_{i=1}^n c_i g_i$ for some
$c_i\in k$ and $g_i\in G(H)$. Then the equation $\Delta(y)=
y\otimes 1+\mu\otimes y$
forces that $y\in k(\mu-1)$.
\end{proof}

\section{Second Lower Bound Theorem}
\label{xxsec2}

In this section we prove Theorem \ref{xxthm0.3}, which is a consequence
of Theorem \ref{xxthm0.2}. A stronger version
will be proved in the next section. Lemmas presented here are also
needed for the next section, and cannot be omitted  even if we skip
Theorem \ref{xxthm0.3}.   If $\GKdim H=\infty$, then
Theorem \ref{xxthm0.3} is vacuous. So we may assume
that $\GKdim H<\infty$. We refer to Section 0 for the definitions
of $W, \Omega, \Gamma$ and $W_{\sqrt{\;}}, \Omega_{\sqrt{\;}},
\Gamma_{\sqrt{\;}}$.

\begin{lemma}
\label{xxlem2.1}
Let $y$ be a skew primitive element not in $C_0$ such that $\gamma(y)$
is defined.
\begin{enumerate}
\item
If $\gamma(y)\in \Gamma\setminus \Gamma_{\sqrt{\;}}$, then $y^n$ is
not skew primitive for any $n>1$.
\item
If $\gamma(y)\in \Gamma_{\sqrt{\;}}$, then $\mu(y)\in W_{\sqrt{\;}}$.
\end{enumerate}
\end{lemma}

\begin{proof} (a) Take $S$ to be the singleton $\{y\}$ and $D=C_0$.
Then (I1.2.1)-(I1.2.3)  hold for $A=k[\mu(y)^{\pm 1}]$. Since
$\gamma(y)\in \Gamma\setminus
\Gamma_{\sqrt{\;}}$ the hypothesis in Theorem \ref{xxthm1.5}(b) holds.
By Theorem \ref{xxthm1.5}(b), $\{y^n\}_{n\geq 0}$ is linearly
independent over $C_0$. Since $\gamma(y)$ is not a root of unity,
for any $n>1$, $\Delta(y^n)\not\in H\otimes k+C_0\otimes H$ by Lemma
\ref{xxlem1.3}(a). Thus $y^n$ is not a skew primitive. The assertion
follows.

(b) Suppose $\gamma(y)\in \Gamma_{\sqrt{\;}}$. Since
$\gamma(y)\neq 1$, replacing $y$ by $y+\alpha(\mu(y)-1)$
for a suitable $\alpha\in k$, we have $\mu(y)^{-1}y\mu(y)=\gamma(y) y$.
Since $\gamma(y)$ is a primitive $n$-th  root of unity for some $n>1$,
Lemma \ref{xxlem1.3}(a) says that $\Delta(y^n)=y^n\otimes 1+\mu(y)^n\otimes
y^n$, which means that $y^n$ is a skew primitive (could be zero).
Therefore $\mu(y)\in W_{\sqrt{\;}}$.
%
\end{proof}

Let $G(H)$ denote the group of all group-like elements in a Hopf algebra
$H$. Recall that $\GKdim H<\infty$ by a general assumption in this section.

\begin{lemma}
\label{xxlem2.2}
Let $y$ be a skew primitive element not in $C_0$ and let $x=\mu(y)$.
Suppose $\gamma(y)$ exists. Assume that $G_0$ is a subgroup of $G(H)$
commuting with $x$. Let $V=k(x-1)+\sum_{g\in G_0} k (g^{-1} y g)$.
\begin{enumerate}
\item
Every $z\in V$ is $(1,x)$-primitive; and $\omega(z)=\omega(y)$
for all $z\in V\setminus k(x-1)$.
\item
If $\gamma(y)\in \Gamma\setminus \Gamma_{\sqrt{\;}}$, then
$\dim V\leq \GKdim H-\GKdim C_0+1$.
\item
Suppose that $V$ is finite dimensional and that $G_0$ is abelian.
Then there is $z\in V\setminus k(x-1)$ such that, either
\begin{enumerate}
\item[(ci)]
for every $g\in G_0$, $g^{-1}zg=\lambda_g z$ for some $\lambda_g\in
k^{\times}$, or
\item[(cii)]
for every $g\in G_0$, $g^{-1}z g=z+\tau_g(x-1)$ for some $\tau_g\in k$.
\end{enumerate}
\item 
If $\gamma(y)$ is not a root of unity, then $\mu(y)$ has infinite order.
\end{enumerate}
\end{lemma}

\begin{proof}
(a) Since $gx=xg$ for all $g\in G_0$,
$g^{-1}y g$ is a $(1,x)$-primitive with $\omega(g^{-1}y g)=\omega(y)$.

(b) Let $S=\{g_i^{-1}yg_i\}_{i=1}^w$ be a finite subset of $V$ which
is linearly independent in the space $V/k(x-1)$. Here $g_i\in G_0$
for all $i=1,\cdots, w$. For different $i$,
we have $\mu(g_i^{-1}y g_i)=x$, and
$$x^{-1}(g_i^{-1}y g_i) x
=g_i^{-1}(x^{-1}yx)g_i=
\gamma(y) (g_i^{-1}y g_i)+\tau(x-1)$$
where $\tau$ is the same as the one in \eqref{I0.3.3}.
Then the hypotheses (I1.2.1)-(I1.2.3) hold for $A=k[x^{\pm 1}]$
and $D=C_0$.
Since $\lambda:=\gamma(y)$ is either 1 or not a root of unity,
Theorem \ref{xxthm1.5}(b) says that $\# S\leq
\GKdim H-\GKdim C_0$. Clearly $V\cap C_0=k(x-1)$. Thus
$$\dim V-1= \dim V/(V\cap C_0)= \# S \leq \GKdim H-\GKdim C_0$$
since $S$ is a basis of $V/(V\cap C_0)$.

(c) First we may assume $x\in G_0$. If not, replace $G_0$ by the
subgroup generated by $G_0$ and $x$ (replacing $G_0$ by this larger
subgroup does not enlarge $V$, because of \eqref{I0.3.3}).
Then $V$ is a $G_0$-module by
conjugation action. Since $G_0$ is abelian and $k$ is algebraically
closed, every finite dimensional simple $G_0$-module
is 1-dimensional. Thus $V$ has a 1-dimensional simple $G_0$-submodule
$kz$. If $z\not\in k(x-1)$, then $kz$ being a simple $G_0$-module is
equivalent to (ci). Otherwise, no element $z\in V\setminus
k(x-1)$ generates a simple $G_0$-submodule. Hence $V$ has a unique
simple $G_0$-submodule $M_0:=k(x-1)$. Note that $g^{-1}(x-1)g=(x-1)$
for all $g\in G_0$, so $M_0$ is the trivial $G_0$-module. Since
$G_0$ is commutative and $V$ has only one simple submodule, every
simple sub-quotient of $V$ must be isomorphic to the simple $M_0$.
Pick $z\in V\setminus k(x-1)$ so that the
submodule $M$ generated by $z$ is 2-dimensional. Then $M/k(x-1)\cong
M_0$, which says that $g^{-1}z g\equiv z$ modulo $k(x-1)$. Hence
$g^{-1}z g=z+\tau_g(x-1)$ for some $\tau_g\in k$.

(d) Let $G_0=\langle \mu(y)\rangle$. It follows from the definition 
that the existence of $\gamma(y)$ implies that $V$ is finite 
dimensional. Applying part (ci) to the cyclic group $G_0$ there is 
a skew primitive $z\in H\setminus C_0$  such that 
$$g^{-1} z g=\lambda(g) z$$
for all $g\in G_0$. It is also clear that $\lambda(\mu(y))=\gamma(y)$. 
Since $\gamma(y)$ is not a root of unity, the image of $\lambda:G_0\to 
k^{\times}$ is infinite. Consequently, $G_0$ is infinite and $\mu(y)$ 
has infinite order.
\end{proof}

\begin{lemma}
\label{xxlem2.3}
Let $\{z_i\}_{i=1}^w$ be a set of skew primitive elements not in $C_0$
such that $\gamma(z_i)$ exists for each $i$.
If the elements $\omega(z_1),\cdots, \omega(z_w)$ are distinct, then
$\{z_i\}_{i=1}^w$ is linearly independent in $H/C_0$.
\end{lemma}

\begin{proof} Suppose $\{z_i\}_{i=1}^w$ is linearly dependent in $H/C_0$. 
Pick a minimal subset, say $\{z_j\}_{j=1}^v$, such that $\sum_{j=1}^v 
a_j z_j =:c\in C_0$ for some scalars $a_j\in k^\times$. Thus $v>1$
since $z_i\not\in C_0$ for any $i$. Applying $\Delta$ to the equation
$\sum_{j=1}^v a_j z_j=c$ we have
$$\Delta(c)=\Delta(\sum_{j=1}^v a_j z_j)=\sum_{j=1}^v a_j z_j\otimes 1+
\sum_{j=1}^v a_j \mu(z_j)\otimes z_j=
c\otimes 1+\sum_{j=1}^v a_j \mu(z_j)\otimes z_j.$$
Hence $\sum_{j=1}^v \mu(z_j)\otimes a_j z_j\in C_0\otimes C_0$.
By the minimality of $v$, $\mu(z_j)=\mu(z_{j'})$ for all $j,j'$.

Set $x=\mu(z_j)$ for all $1\leq j\leq v$.
Applying the conjugation by $x$ to the equation $\sum_{j=1}^v a_j z_j=c$,
we obtain $\sum_{j=1}^v \gamma(z_j) a_j z_j=-\sum_{j=1}^v \tau_j(x-1)+
x^{-1}c x\in C_0$ for some $\tau_j\in k$. Using the minimality of $v$, 
$\gamma(z_j)=
\gamma(z_{j'})$ for all $j,j'$. Thus we obtain a contradiction. The
assertion follows.
\end{proof}

\begin{theorem}
\label{xxthm2.4}
Let $\{y_i\}_{i=1}^w$ be a set of skew primitive elements not in $C_0$ 
such that
$\omega(y_1),\cdots, \omega(y_w)$ are defined and distinct elements in
$\Omega\setminus \Omega_{\sqrt{\;}}$. If the subgroup $G_0$ generated by
$\{\mu(y_i)\}_{i=1}^w$ is abelian, then $\GKdim H\geq \GKdim C_0+w$.
\end{theorem}

\begin{proof} By Lemma \ref{xxlem2.2}(a,c), for each $i$, there is a
$z_i$ in $k(\mu(y_i)-1)+\sum_{g\in G_0} k g^{-1} y_i g$ but not 
in $C_0$ such that
$$\omega(z_i)=\omega(y_i),$$
and, for every $g\in G_0$,
$$g^{-1} z_i g=\lambda_{ig} z_i+\tau_{ig}(\mu(y_i)-1)$$
for some $\lambda_{ig}\in k^\times, \tau_{ig}\in k$.
Let $A=k G_0$ and $D=C_0$. Then (I1.2.1) is clear and (I1.2.2) follows from
Lemma \ref{xxlem2.3} for the set $\{z_i\}_{i=1}^w$.
(I1.2.3) is a consequence of Lemma \ref{xxlem2.2}(c) as we have seen
already. By hypothesis each $\lambda_i:=\gamma(z_i)$ is either 1 or
not a root
of unity. Therefore $\GKdim H\geq \GKdim C_0+w$ by applying Theorem
\ref{xxthm1.5}(b) to the set $\{z_i\}_{i=1}^w$.
\end{proof}

The next lemma is a result of \cite{Zhu}. As before we assume that
$\GKdim H<\infty$ which is one of the hypotheses in
\cite[Theorem 1.2]{Zhu}.

\begin{lemma}\cite[Theorem 1.2]{Zhu}
\label{xxlem2.5}
Let $y$ be a skew primitive element not in $C_0$ with $g=\mu(y)$.
Then there is a skew primitive element $z=\sum_{i=0}^n b_i g^{-i} y g^i\in
H\setminus C_0$, where $b_i\in k$, such that 
$g^{-1}z g=\lambda z+\tau(g-1)$ for some
$\lambda\in k^{\times}$ and $\tau\in k$. Further, if $\lambda\neq 1$,
then there is $z'=z+\alpha(g-1)$ for a suitable $\alpha\in k$
such that $g^{-1}z'g=\lambda z'$.
\end{lemma}

\begin{proof} In \cite[Theorem 1.2]{Zhu} $H$ is assumed to be pointed,
but the statement is valid without this hypothesis.
The first assertion is equivalent to \cite[Theorem 1.2]{Zhu}.
If $\lambda\neq 1$, take $\alpha=(\lambda-1)^{-1}\tau$. Then $z'=z+
\alpha (g-1)$ is a $(1,g)$-primitive element satisfying
$g^{-1}z'g=\lambda z'$.
\end{proof}

Now we are ready to prove Theorem \ref{xxthm0.3}.

\begin{proof}[Proof of Theorem 0.3]
Pick any finite subset $\{\mu(y_i)\}_{i=1}^w$ of
$W\setminus W_{\sqrt{\;}}$ where each $y_i$ is a skew primitive
not in $C_0$. By Lemma \ref{xxlem2.5},
for each $i$ there is a skew primitive $y'_i$ not in $C_0$ such that
$g_i:=\mu(y'_i)=\mu(y_i)$ and that $\gamma(y'_i)$ is
defined. By Lemma \ref{xxlem2.1}(b), $\gamma(y'_i)$
is not a root of unity or 1. Hence $\omega(y'_i)\in
\Omega\setminus \Omega_{\sqrt{\;}}$. The assertion follows
from Theorem \ref{xxthm2.4}.
\end{proof}

Theorem \ref{xxthm2.4} shows in fact
that if $\langle W\setminus W_{\sqrt{\;}}
\rangle$ is abelian, then
$$\GKdim H\geq \GKdim C_0+\#(\Omega\setminus \Omega_{\sqrt{\;}}).
$$
There is also an inequality
$$\GKdim H\geq \GKdim C_0+ \#( W')$$
for any $W'\subset W\setminus W_{\sqrt{\;}}$ such that $\langle W'\rangle$
is abelian.

Suppose there is a surjective Hopf algebra morphism
$\pi: H\to C_0$ such that the restriction to $C_0$ is
the identity. Let $A$ be the subalgebra of $H$ generated
by all skew primitive elements in $\ker \pi$
and let $G_W$ be the sub-semigroup of $G(H)$ generated
by $\mu(y)$ for all skew primitive elements $y\in A$.
We do not assume that $G_W$ is abelian.

\begin{lemma}
\label{xxlem2.6}
Suppose there is a surjective Hopf algebra morphism
$\pi: H\to C_0$ such that the restriction to $C_0$ is
the identity. Let $A$ be defined as above.
\begin{enumerate}
\item
$H=R\# C_0$ where $R$ is the ring of right coinvariants of $\pi$.
Then $A$ is a subalgebra of $R$ and
$$\GKdim H\geq \GKdim R+\GKdim C_0\geq \GKdim A+\GKdim
C_0.$$
\item
Assume that $A$ is a domain. Then
$\GKdim A\geq \GKdim k G_W$. As a consequence,
$$\GKdim H\geq \GKdim kG_W+\GKdim C_0.$$
\end{enumerate}
\end{lemma}

\begin{proof} (a) By \cite[Theorem 7.2.2]{Mo}, $H$ is isomorphic
to a crossed product $R\#_{\sigma} C_0$ as algebras and by
\cite[Proposition 7.2.3]{Mo}, $\sigma$ is trivial. Hence
$H=R\# C_0$ where $R$ is the ring of right coinvariants of $\pi$.
It is clear that every skew primitive element in $\ker \pi$ is in $R$.
Therefore $A\subset R$.

Since $H=R\# C_0$, $\GKdim H\geq \GKdim R+\GKdim C_0$.
The assertion follows by the fact $A\subset R$.

(b) Define a map $\rho: A\to C_0\otimes H$ to be the
composition $(\pi\otimes Id_H)\circ \Delta$. Since $\rho(y)
\in kG_W\otimes A$ for all skew primitive elements $y\in A$ and
since $A$ is generated by these $y$'s, the image of $\rho$
is in $kG_W\otimes A$. Consequently, $(A,\rho)$
is a left $kG_W$-comodule algebra. This means that $A$ is
a $G_W$-graded algebra.  Let $f: A\to C_0$ be the
map sending any nonzero homogeneous element $h\in A$ to its degree,
for example, sending $y_1\cdots y_n$ to $\mu(y_1)\cdots \mu(y_n)$.
Since $A$ is a domain, $f$ is multiplicative.

Pick any finite dimensional space $V=k+\sum_{i=1}^m k \mu(y_i)$ of
$k G_W$ where the $y_i$ are skew primitive elements in $A$,
let $W=\{1\}\cup \{y_i\}_{i=1}^w$. Then $\dim (kW)^n\geq
\#(f(W^n))\geq  \# (f(W))^n\geq \dim V^n$ for all $n$. Hence
$\GKdim A\geq \GKdim k G_W$.
\end{proof}

The next example shows why we need to remove $W_{\sqrt{\;}}$
from $W$ (or remove $\Gamma_{\sqrt{\;}}$ from $\Gamma$)
in the lower bound theorems.

\begin{example}
\label{xxex2.7}
Let $B$ be the Hopf algebra $B(1,1,p_1,\cdots,p_s,q)$ defined in
\cite[Construction 1.2]{GZ2}. This is a finitely generated, noetherian,
pointed Hopf domain of GK-dimension  2. By
\cite[Construction 1.2]{GZ2} $B$ is generated by $x,x^{-1},
y_1,\cdots,y_s$ where $x$ is a group-like element and $y_i$'s are
skew primitive elements. Let $z=y_1^{p_1}$. Then $z=y_{j}^{p_j}$
for all $j$ and it is a central skew primitive element. Let
$H=B/(z, x^m-1)$ where $m=\prod_i p_i$. Then $H$ is a finite
dimensional pointed Hopf algebra of GK-dimension  0 and $C_0=
k[x,x^{-1}]/(x^m-1)$ has GK-dimension 0.

By \cite[Construction 1.2]{GZ2},
$W=W_{\sqrt{\;}}=\{x^{m_i}\}_{i=1}^s$ where $m_i=m/p_i$,
$\Gamma=\Gamma_{\sqrt{\;}}=\{q^{-m_i^2}\}_{i=1}^s$, $\Omega=
\Omega_{\sqrt{\;}}=\{(x^{m_i},q^{-m_i^2})\}_{i=1}^s$, and
$Z=Y_{\sqrt{\;}}+C_0$ and $Y_{\sqrt{\;}}=\sum_{i=1}^s k y_i$. Thus
$$\#(W_{\sqrt{\;}})=\#(\Gamma_{\sqrt{\;}})=
\#(\Omega_{\sqrt{\;}})=\dim Y_{\sqrt{\;}}=s$$
which can be arbitrarily large.
\end{example}

\section{Third lower bound theorem}
\label{xxsec3}

The first half of this section concerns some preliminary analysis
of Hopf algebras with exponential growth and the proof of Proposition
\ref{xxprop0.5}. The proof of the third lower bound theorem is given
at the end of the section.

Let $(G_0,\times)$ be a multiplicative abelian group and $\Lambda:=
\{\lambda_1,\cdots,\lambda_v\}$ be a list of 1-dimensional group
representations of $G_0$ for some $v>1$. Note that this list is allowed
to have repetitions. When some
$\lambda_i$ is the trivial representation of $G_0$ (namely,
$\lambda_i(g)=1$ for all $g\in G_0$), then we also need a group 
homomorphism $\tau_i: (G_0,\times)\to (k,+)$ (which must be zero if
$G_0$ is torsion since ${\text{char}}\; k=0$).
When $\lambda_i$ is not trivial, we set $\tau_i=0$.

Now pick a list of elements $\mu:=\{\mu_1,\cdots,\mu_v\}$ in $G_0$
(again allowing repetitions).
Let $K:=K(\Lambda,\mu)$ be the Hopf algebra generated as an algebra
by the elements
in the abelian group $G_0$ and a set of skew primitive elements
$y_1,\cdots,y_v$ subject to the relations within $G_0$ and the
following additional relations between $G_0$ and $\{y_i\}_{s=1}^v$,
$$\begin{aligned}
y_i g &=\lambda_{i}(g) g y_i+\tau_i(g) g(\mu_i-1),
{\text{ for all $i$ and all $g\in G_0$}}.
\end{aligned}
$$
The coalgebra structure of $K$ is determined by
$$\begin{aligned}
\Delta(g)=g\otimes g, &\quad
\epsilon(g)=1, \quad {\text{for all $g\in G_0$}},\\
\Delta(y_i)=y_i\otimes 1+\mu_i\otimes y_i, &\quad \epsilon(y_i)=0,
\quad {\text{for all $i=1,\cdots,v$}}.
\end{aligned}
$$
And the antipode of $K$ is determined by
$$\begin{aligned}
S(g)&=g^{-1}, \quad {\text{for all $g\in G_0$}},\\
S(y_i)&=-\mu_i^{-1} y_i,\quad {\text{for all $i=1,\cdots,v$}}.
\end{aligned}
$$
Let $\lambda_{ij}=\lambda_i(\mu_j)$ for all $i,j$ and let
$\Lambda_M$ be the  $v\times v$-matrix $(\lambda_{ij})$.

By Remark \ref{xxrem1.2}, the total $y$-degree and the
multi-$y$-degree are defined for elements in $K$. For example,
the (filtered) multi-$y$-degree of $g y_{3} y_2$ is $(0,1,1,0,\cdots,0)
\in {\mathbb N}^v$.

Let $F$ be a nonzero skew primitive element in $K$ with
total $y$-degree $z\ge 2$.
Write $F=\sum c_{h,(i_s)} hy_{i_1}y_{i_2}\cdots y_{i_n}$ where
$h\in G_0$ and $0\neq c_{h,(i_s)}\in k$. A {\it term} of $F$ means
a nonzero monomial $c_{h,(i_s)} hy_{i_1}y_{i_2}\cdots y_{i_n}$
appearing in $F$.

\begin{lemma}
\label{xxlem3.1} Let $K:=K(\Lambda,\mu)$ be defined as above.
\begin{enumerate}
\item
$K$ has a $k$-linear basis
$$\{gy_{i_1}y_{i_2}\cdots y_{i_s}\}$$
where $g\in G_0$, $i_1,\cdots, i_s\in \{1,\cdots, v\}$.
As a consequence, $K$ contains a free subalgebra
$k\langle y_1,y_2\rangle$ and has exponential growth.
\item
The coradical of $K$ is $kG_0$.
\item
If $F$ is a skew primitive element of total $y$-degree $z\ge 2$,
then for any term of $F$ with multi-$y$-degree $(N_1,\cdots, N_v)$ and
$\sum_i N_i=z$,
$$\prod_{i=1}^v (\lambda_{ii})^{ N_i(N_i-1)}\prod_{i<j}
(\lambda_{ij}\lambda_{ji})^{N_iN_j}=1.$$
\end{enumerate}
\end{lemma}

\begin{proof} (a) The first assertion follows from Bergman's Diamond
Lemma \cite[Theorem 1.2]{Be}. Consequently, $K$ contains the free
algebra of rank 2, $k\langle y_1, y_2\rangle$. Therefore $K$ has
exponential growth.

(b) By definition, $\Delta$ is compatible with filtrations defined
in Remark \ref{xxrem1.2}. Hence
$\Delta$ is a homomorphism of filtered algebras. So
every group-like element must have total $y$-degree 0.
The assertion follows.

(c) Let $F=\sum c_{h,(i_s)} hy_{i_1}y_{i_2}\cdots y_{i_n}$ with
coefficients $c_{h,(i_s)}\neq 0$.
 For simplicity, let $ldt$ denote any linear combination of
monomials of total $y$-degree less than $z$. Then we can write
$F=\sum c_{h,(i_s)} hy_{i_1}y_{i_2}\cdots y_{i_z}+ldt$.
Since $\Delta(F)=F\otimes 1+\mu(F)\otimes F$, $h=1$ for
terms with total degree $z$. Pick any term of $y$-degree $z$ in $F$,
say $c_{1,(i_s)} y_{i_1}y_{i_2}\cdots y_{i_z}$, and let
$(N_1,\cdots, N_v)$ be its multi-$y$-degree.

Since $F$ is skew primitive, $S(F)=-\mu(F)^{-1} F$.
Since $S(y_i)= -\mu_i^{-1} y_i$, we have
$$\begin{aligned}
S(c_{1,(i_s)}y_{i_1}\cdots y_{i_z})&=c_{1,(i_s)} (-\mu_{i_z}^{-1}
y_{i_z})\cdots (-\mu_{i_1}^{-1} y_{i_1})\\
&=c_{1,(i_s)} (-1)^z \prod_{s>t} \lambda_{i_s i_t}^{-1} \mu^{-1}
y _{i_z}\cdots y_{i_1}+ldt
\end{aligned}
$$
where $\mu=\prod_{s=1}^z \mu_{i_s}$. Since $S(F)=-\mu(F)^{-1} F$,
$\mu=\mu(F)$ and $F$ contains a nonzero term of the form
$c'_{(i_s)}y _{i_z}\cdots y_{i_1}$. The same computation shows that
$$\begin{aligned}
S(c'_{(i_s)}y_{i_z}\cdots y_{i_1})&=c'_{(i_s)} (-\mu_{i_1}^{-1} y_{i_1})\cdots
(-\mu_{i_z}^{-1} y_{i_z})\\
&=c'_{(i_s)} (-1)^z \prod_{a<b} \lambda_{i_a i_b}^{-1} \mu^{-1}
y _{i_1}\cdots y_{i_z}+ldt.
\end{aligned}
$$
Comparing the coefficients in the terms $\mu^{-1}y _{i_z}\cdots y_{i_1}$
and $\mu^{-1}y _{i_1}\cdots y_{i_z}$ in the equation $S(F)=-\mu^{-1}F$,
we have
$$-c'_{(i_s)}=c_{1,(i_s)} (-1)^z\prod_{s>t} \lambda_{i_s i_t}^{-1},
\quad
-c_{1,(i_s)}=c'_{(i_s)}(-1)^z \prod_{a<b} \lambda_{i_a i_b}^{-1}.$$
Since $c_{1,(i_s)}$ and $c'_{(i_s)}$ are nonzero, the above two equations
imply
$$\prod_{s>t} (\lambda_{i_s i_t}\lambda_{i_t i_s})=1$$
or
\begin{equation}
\label{I3.1.1}\tag{I3.1.1}
\prod_{\{s\neq t\}\subset \{1,2,\cdots,z\}}
(\lambda_{i_s i_t}\lambda_{i_t i_s})=1.
\end{equation}
We know the monomial $y_{i_1}\cdots y_{i_n}$ contains $N_i$ copies of
$y_i$ for all $i=1,\cdots,v$. Hence equation \eqref{I3.1.1} is in fact
$$\prod_{i=1}^v (\lambda_{ii})^{N_i(N_i-1)}\prod_{i<j}
(\lambda_{ij}\lambda_{ji})^{N_iN_j}=1.$$
\end{proof}

There is a slight modification of Lemma \ref{xxlem3.1}. Suppose
$\lambda_{11}$ is a primitive $p_1$-th root of unity for some $p_1>1$.
Recycle most of the notations before Lemma \ref{xxlem3.1}.
Let $L:=L(\Lambda,\mu, p_1)$ be the Hopf algebra generated as an algebra by the
abelian group $G_0$ and $y_1,\cdots,y_v$ subject to the relations
within $G_0$ and the following additional relations between $G_0$
and $\{y_i\}_{s=1}^v$
$$\begin{aligned}
y_i g &=\lambda_{i}(g) g y_i+\tau_i(g) g(\mu_i-1),
{\text{ for all $i$ and all $g\in G_0$}},\\
y_1^{p_1}&=\beta(\mu_1^{p_1}-1),
\quad {\text{ for some $\beta\in k$}}.
\end{aligned}
$$
The coalgebra structure of $L$ is determined by
$$\begin{aligned}
\Delta(g)=g\otimes g, &\quad
\epsilon(g)=1, \quad {\text{for all $g\in G_0$}},\\
\Delta(y_i)=y_i\otimes 1+\mu_i\otimes y_i, &\quad \epsilon(y_i)=0,
\quad {\text{for all $i=1,\cdots,v$}}.
\end{aligned}
$$
And the antipode of $L$ is determined by
$$\begin{aligned}
S(g)&=g^{-1}, \quad {\text{for all $g\in G_0$}},\\
S(y_i)&=-\mu_i^{-1} y_i,\quad {\text{for all $i=1,\cdots,v$}}.
\end{aligned}
$$

Define $\Lambda_M:=(\lambda_{ij})=(\lambda_i(\mu_j))$.
The total $y$-degree and the multi-$y$-degree are defined
as before.

\begin{lemma}
\label{xxlem3.2} Let $L:=L(\Lambda,\mu,p_1)$ be defined as above.
Suppose either $\beta=0$ or $\lambda_1(g)^{p_1}=1$ for all
$g\in G_0$.
\begin{enumerate}
\item
$L$ has a $k$-linear basis
$$\{gy_{i_1}y_{i_2}\cdots y_{i_s}\}$$
where $g\in G_0$, $i_1,\cdots, i_s\in \{1,\cdots, v\}$
and there is no $u$ such that $i_u=i_{u+1}=
\cdots i_{u+p_1-1}=1$. As a consequence, $L$ contains a free
subalgebra $k\langle y_1y_2, y_1 y_2^2\rangle$ and has exponential
growth.
\item
The coradical of $L$ is $kG_0$.
\item
If $F$ is a skew primitive element of total $y$-degree $z\ge 2$,
then for any term of $F$ with multi-$y$-degree $(N_1,\cdots, N_v)$
and $\sum_i N_i=z$,
$$\prod_{i=1}^v (\lambda_{ii})^{N_i(N_i-1)}\prod_{i<j}
(\lambda_{ij}\lambda_{ji})^{N_iN_j}=1.$$
\end{enumerate}
\end{lemma}

\begin{proposition}
\label{xxprop3.3}
Let $H$ be a Hopf algebra and $y_1,y_2$ be two skew primitive
elements linearly independent in $H/C_0$. Suppose that
\begin{enumerate}
\item[]
\begin{enumerate}
\item[(i)]
there is a group-like element $x$ and $d_1,d_2\in {\mathbb Z}$ such that
$\mu(y_i)=x^{d_i}$ for $i=1,2$, and 
\item[(ii)]
there are two scalars $q_1,q_2\in k^{\times}$ such that
$y_i x=q_i xy_i$ for $i=1,2$.
\end{enumerate}
\end{enumerate}
\begin{enumerate}
\item
If $x$ has infinite order and $H$ does not contain a free algebra of 
rank 2, then
\begin{equation}
\label{I3.3.1}\tag{I3.3.1}
q_1^{d_1(M_1(M_1-1))+d_2(M_1M_2)}
q_2^{d_2(M_2(M_2-1))+d_1(M_1M_2)}=1
\end{equation}
for some integers  $M_1,M_2\geq 0$ satisfying $M_1+M_2\geq 2$.
\item
Assume that one of the following holds:
\begin{enumerate}
\item[(1)]
$q_1=q_2$ is not a root of unity and $d_1d_2> 0$;
\item[(2)]
$q_1^{d_1}=1$ and $q_2$ is not a root of unity and $d_1d_2> 0$;
\item[(3)]
$q_1^{d_1}\neq 1$ is a root of unity and $q_2$ is not a root of unity and
$d_1d_2>0$;
\item[(4)]
the group $\langle q_1, q_2\rangle\subset k^\times$ is free abelian
of rank 2, $d_1d_2\neq 0$.
\end{enumerate}
Then $H$ contains a free subalgebra of rank 2. Consequently, $H$ has
exponential growth.
\end{enumerate}
\end{proposition}

\begin{proof} (a) Let $\mu_i=x^{d_i}$ and $\lambda_{ij}=q_i^{d_j}$. Then
$y_i \mu_j=\lambda_{ij} \mu_j y_i$ and $\gamma(y_i)=\lambda_{ii}$ 
for all $i,j\in \{1,2\}$. 

Let $H_0$ be the Hopf subalgebra generated by $x,x^{-1},y_1,y_2$.
Let $G_0=\langle g \rangle\cong {\mathbb Z}$ and 
let $\lambda_i(g^n)= q_i^n$ for $i=1,2$ and all $n$.
Let $\Lambda=\{\lambda_1,\lambda_2\}$ and $\mu=\{g,g\}$. 
Then there is a surjective Hopf algebra homomorphism 
$\phi: K:=K(\Lambda, \mu)\to H_0$ sending $g\mapsto x$
and $y_i\mapsto y_i$ for $i=1,2$, where we choose $\tau_i=0$. 
By Lemma \ref{xxlem3.1}(a) $K$ contains a free algebra of rank 2.  
If $H$ does not contain a free algebra of rank 2,
then $K\to H_0$ is not injective. By \cite[Theorem 5.3.1]{Mo},
there is a nonzero skew primitive element $F\in K$ such that
$\phi(F)=0$. Since $\phi$ is injective on skew primitive elements
of $y$-degree $\le 1$, $F$ has total $y$-degree $z\ge 2$. By Lemma
\ref{xxlem3.1}(c), for any term of $F$ with multi-$y$-degree $(M_1,M_2)$
and $M_1+M_2= z$, we have the following, 
$$
(\lambda_{11})^{M_1(M_1-1)}(\lambda_{22})^{M_2(M_2-1)}
(\lambda_{12}\lambda_{21})^{M_1M_2}=1
$$
or equivalently,
$$(q_1^{d_1})^{M_1(M_1-1)}(q_{2}^{d_2})^{M_2(M_2-1)}
(q_1^{d_2}q_2^{d_1})^{M_1 M_2}=1.$$
This can be simplified to
$$
q_1^{d_1(M_1(M_1-1))+d_2(M_1M_2)}
q_2^{d_2(M_2(M_2-1))+d_1(M_1M_2)}=1
$$
which is \eqref{I3.3.1}.

(b) Assume $H$ does not contain a free algebra of rank 2 and 
we will obtain a contradiction. If one of the hypotheses holds, then 
$x$ has infinite order in $G(H)$
by Lemma \ref{xxlem2.2}(d). Therefore we can apply part (a). 

In case (b1), \eqref{I3.3.1} implies that 
$$d_1(M_1(M_1-1))+d_2(M_1M_2)+
d_2(M_2(M_2-1))+d_1(M_1M_2)=0.$$
This is impossible since $d_1d_2>0$ and $M_1+M_2\geq 2$.
Therefore $H$ contains a free algebra of rank 2.

A similar argument works for case (b4).

In case (b2), equation \eqref{I3.3.1} implies that
$(q_2^{d_2(M_2(M_2-1))+d_1(M_1M_2)})^{d_1}=1$, or
$$d_2(M_2(M_2-1))+d_1(M_1M_2)=0$$
because $q_2$ is not a root of unity.
Since $d_1d_2>0$, the only solution is $M_2=0$ and $M_1=z\geq 2$.
Thus we have $F=cy_1^{z}+ldt$ for some $c\in k^{\times }$.
By Lemma \ref{xxlem1.3} and the fact that $\lambda_{11}=q_1^{d_1}=1$,
$F$ cannot be skew primitive for any $z\geq 2$. So case (b2) has
been taken care of.

It remains to consider case (b3). Suppose $q_1^{d_1}$ is a primitive
$p_1$-th root of unity. Then $y_1^{p_1}$ is a skew primitive element.
If $y_1^{p_1} \not\in C_0$, then $\{y_1^{p_1},y_2\}$
is linearly independent in $H/C_0$. Note that if $\alpha y_1^{p_1}+\beta
y_2\in C_0$, then $x^{-1}(\alpha y_1^{p_1}+\beta
y_2)x \in C_0$, which would imply $y_1^{p_1}, y_2\in C_0$ because
$q_1^{p_1}\neq q_2$.
The assertion follows from case (b2) applied to $\{y_1^{p_1},y_2\}$.
If  $y_1^{p_1}\in C_0$, then $y_1^{p_1}=\beta(\mu_1^{p_1}-1)$
for some $\beta$. Replacing $K$ by $L$ in the above argument, 
\eqref{I3.3.1} holds again.

Since $q_1$ is a root of unity, we have
$$ (q_2^p)^{d_2(M_2(M_2-1))+d_1(M_1M_2)}=1$$
for some $p>1$. Since $q_2$ is not a root of unity,
$$d_2(M_2(M_2-1))+d_1(M_1M_2)=0.$$
Since $d_1d_2>0$, the only solution is $M_2=0$ and $M_1=z\geq 2$.
Now $F=cy_1^{z}+ldt$, where $c\in k^{\times }$ and $z< p_1$.
By Lemma \ref{xxlem1.3}, $F$ cannot be skew primitive. 
This is a contradiction. 
\end{proof}

%
%
%
%

\begin{corollary}
\label{xxcor3.4}
Suppose $H$ has subexponential growth. Let $y$ be a skew primitive 
element not in $C_0$ such that $\gamma(y)$ is defined and is not a root of unity. 
Let $G_0$ be a finitely generated abelian subgroup of $G(H)$ 
containing $\mu(y)$ (which has infinite order automatically). Then 
$V:=k(\mu(y)-1)+\sum_{g\in G_0} k(g^{-1} yg)$ is 2-dimensional. As 
a consequence, there is a group representation $\lambda: G_0\to k$ 
such that 
$$g^{-1} y' g=\lambda(g)y'$$
for all $g\in G_0$, where $y'=y+\alpha(\mu(y)-1)$ for some 
$\alpha\in k$.
\end{corollary}

\begin{proof} Since $\gamma(y)$ is not a root of unity, we may assume
that $y \mu(y)=\gamma(y) \mu(y) y$ after replacing $y$ by
$y+\alpha(\mu(y)-1)$ for some $\alpha\in k$. Let $g\in G_0$. Let 
$y_1=y$ and $y_2=g^{-1}yg$. Then $\gamma(y_1)=\gamma(y_2)$ and it is 
not a root of unity. By Proposition \ref{xxprop3.3}(b1), $y_1$ and 
$y_2$ are not linearly independent in $H/C_0$. The assertion that $\dim V=2$ follows 
by applying Lemma \ref{xxlem1.6}. The consequence follows from 
Lemma \ref{xxlem2.2}(ci). 
\end{proof}

\begin{proof}[Proof of Proposition \ref{xxprop0.5}]
We prove the assertion by contradiction. So we assume that
$H$ has subexponential growth.

Pick a pair of skew primitive elements $(y_1,y_2)$ such that the 
subgroup of $\langle\Gamma\rangle$ generated by $\{\gamma(y_1),
\gamma(y_2)\}$ has rank $2$. Let $\lambda_{ii}=\gamma(y_i)$ and 
$g_i=\mu(y_i)$ for $i=1,2$. Since $\langle W\rangle$ is abelian, 
the subgroup $G_0:=\langle g_1, g_2\rangle$ is abelian. By 
Corollary \ref{xxcor3.4}, we may further assume that, for 
$i=1,2$ and $j=1,2$, $g_j^{-1} y_i g_j=\lambda_{ij} y_i$ for 
some $\lambda_{ij}\in k^\times$. Since $H$ does not contain a 
free subalgebra of rank $2$, the proofs of Proposition 
\ref{xxprop3.3} and \eqref{I3.3.1} show that
\begin{equation}
\label{I3.4.1}
\tag{I3.4.1}
(\lambda_{11})^{M_1(M_1-1)}
(\lambda_{22})^{M_2(M_2-1)}
(\lambda_{12}\lambda_{21})^{M_1M_2}=1
\end{equation}
for some non-negative $M_1,M_2$ with $M_1+M_2\geq 2$.

Since $G_0$ is a finitely generated subgroup of $\langle W\rangle$ 
and since $\langle W\rangle$ is abelian and torsionfree of rank 1, 
$G_0$ is isomorphic to ${\mathbb Z}$. Therefore there is an 
$x\in G_0$ such that $g_1=x^a$ and $g_2=x^b$ for some nonzero 
integers $a,b$. Consequently, $g_2^a=g_1^b$. Thus the equation 
$g_j^{-1} y_i g_j=\lambda_{ij} y_i$ implies that $\lambda_{i2}^a
=\lambda_{i1}^b$ for all $i=1,2$. Then \eqref{I3.4.1} implies that
$$(\lambda_{11})^{ab M_1(M_1-1)}
(\lambda_{22})^{ab M_2(M_2-1)}
(\lambda_{11})^{b^2 M_1 M_2}(\lambda_{22})^{a^2 M_1M_2}=1$$
Since the rank of $\langle \lambda_{11}, \lambda_{22}\rangle$ is 2,
we have
$$ab M_1(M_1-1)+b^2 M_1M_2=
ab M_2(M_2-1)+a^2 M_1M_2=0$$
or
$$a (M_1-1)+b M_2=b (M_2-1)+a M_1=0.$$
Since $a,b$ are nonzero, this means that $(M_1-1)(M_2-1)-M_1M_2=0$.
This is impossible when $M_1+M_2\geq 2$, which yields a
contradiction.
\end{proof}

The rest of this section is devoted to the proof of Theorem
\ref{xxthm0.4}. The next definition was given in the introduction,
but maybe it should be reviewed here. Let
$y$ be a $(1,g)$-primitive element in a Hopf
algebra $H$. Let $T_{g^{-1}}$ denote the inverse conjugation by
$g$, namely, $T_{g^{-1}}:  a\to g^{-1} ag$.

\begin{definition}
\label{xxdefn3.5}
Let $y$ be a $(1,g)$-primitive element of $H$ not in $C_0$. A nonzero
scalar $\lambda$ is called the commutator of $y$ of level $n$ if
$(T_{g^{-1}}-\lambda Id_H)^n(y)\in C_0$ and
$(T_{g^{-1}}-\lambda Id_H)^{n-1}(y)\not\in C_0$. In this case we write
$\gamma(y)=\lambda$. When $n=1$, $\gamma(y)$ is the commutator of $y$
defined as in \eqref{I0.3.2} or equivalently in \eqref{I0.3.3}.
\end{definition}

In general the commutator of $y$ may not exist. We also need a
generalization of Definition \ref{xxdefn3.5}. Recall that 
$W_{\times}$ is the subset of $W$ consisting of weights $\mu(y)$ 
such that the commutator of $y$ is either 1 or not a root of unity. 
Throughout the rest of the section let $G_0$ be the 
subgroup $\langle W_{\times}\rangle$ and suppose that $G_0$ is
abelian. A 1-dimensional representation of $G_0$ is equivalent to
a multiplicative map $\lambda: G_0\to k^{\times}$. Let $G_0^*$
denote the set of 1-dimensional representations of $G_0$, which
is also called the character group of $G_0$.

\begin{definition}
\label{xxdefn3.6}
Let $y$ be a skew primitive element in $H\setminus C_0$ and let
$\lambda\in G_0^*$. We say $\lambda$ is the {\it generalized
commutator} of $y$ of level $n$ if there is an $n$ such that
$(T_{g^{-1}}-\lambda(g) Id_H)^n (y)\in C_0$ for all $g\in G_0$ and
$(T_{g^{-1}}-\lambda(g) Id_H)^{n-1} (y)\not\in C_0$ for some $g\in G_0$.
If $n=1$, $\lambda$ is called the {\it generalized commutator} of
$y$.
\end{definition}

\begin{lemma}
\label{xxlem3.7} 
In parts (a) and (c) suppose that $\GKdim H<\infty$.
\begin{enumerate}
\item
Every skew primitive element $y\in H$ is a linear combination of
skew primitives with commutator of finite level and weight $\mu(y)$.
\item
Suppose $V$ is a conjugation $G_0$-stable finite dimensional
subspace spanned by skew primitives with their weights in $G_0$. 
Then every element $y\in V$ is a linear combination of skew 
primitives with generalized commutator of finite level.
\item
Every skew primitive $y\in H$ with commutator of level 1 such that
$\gamma(y)\in \Gamma\setminus \Gamma_{\sqrt{\;}}$
is a linear combination of skew primitives with generalized
commutator of finite level. Further, each nonzero summand of the 
linear combination has weight commutator $\omega(y)$. 
\end{enumerate}
\end{lemma}

\begin{proof} (a) If $y\in C_0$, then $y=\alpha (\mu(y)-1)$ for some
$\alpha\in k$ [Lemma \ref{xxlem1.6}] and the commutator of $y$ has 
level 0 by definition.

If $y\not\in C_0$, then, by Lemma \ref{xxlem2.5},
$V:=\sum_{n\in {\mathbb Z}} k (g^{-n} yg^n)+k(g-1)$ is finite dimensional
over $k$, where $g=\mu(y)$.
Then $T_{g^{-1}}$ acts on $V$ as an invertible linear map.
Pick a basis of $V$ so that the presentation of $T_{g^{-1}}$
with respect to the basis is in the Jordon canonical form.
Then each basis element is a skew primitive element
with commutator of finite level. The assertion follows.

(b) Write $V=\bigoplus_{j=1}^m V_j$ where each $V_j$ is spanned by 
skew primitives with weight $g_j$ for distinct group-like
elements $g_1,\cdots, g_m\in G_0$. Since $G_0$ is abelian,
each $V_j$ is conjugation $G_0$-stable. Passing from $V$ to $V_j$
we may assume that each element in $V$ is a skew primitive
of weight $g$. 

For every $h\in G_0$, $T_{h^{-1}}$ acts on $V$ as an invertible
linear map. It is clear that $V$ is a finite dimensional $kG_0$-module.
Since $kG_0$ is commutative, every finite dimensional simple
$G_0$-module is 1-dimensional and $\Ext^1_{kG_0}(S,S')=0$ if $S$ and
$S'$ are distinct simple modules over $kG_0$. Then $V$ is a finite
direct sum of submodules $V_i$ so that the support of each $V_i$ is
a single closed point of $\Spec kG_0$. This closed point
corresponds to a 1-dimensional $G_0$-representation $\lambda_i$.
Fix any $i$, every element in $V_i$ has a generalized commutator
$\lambda_i$ with level no more than $\dim V_i$.

(c) Let $V$ be the vector space spanned by all skew
primitive elements $z$ with commutator of level 1
such that $\omega(z)=\omega(y)$. Pick any finite
set $\{z_1,\cdots,z_w\}$ which is linearly independent
in $V/(V\cap k(\mu(y)-1))$. Then (I1.2.1)-(I1.2.3)
hold for $D=C_0$ and $A=k[\mu(y)^{\pm 1}]$. By Theorem
\ref{xxthm1.5}(b), $w\leq \GKdim H$. Thus $V$ is
finite dimensional. Clearly, $T_{g^{-1}}$ stabilizes $V$ for all
$g\in G_0$. The first assertion follows from part (b). The
final assertion is clear since every element in $V$
has weight commutator equal $\omega(y)$. 
\end{proof}

We need to introduce some conventions.
Let $\mu\in W$. Define $P_{\mu}$ to be the $k$-linear space
spanned by all skew primitives $y\in H$ with $\mu(y)=\mu$.
Let $\gamma$ be a nonzero scalar and let $P_{\mu,\gamma,n}$
be the $k$-linear space spanned by all skew primitives $y\in P_{\mu}$
with commutator $\gamma$ of level no more than $n$.
Let $P_{\mu,\gamma,*}=\sum_{n\geq 0}P_{\mu,\gamma,n}$,
$P_{*,*,n}=\sum_{\mu,\gamma}P_{\mu,\gamma,n}$ and
$P_{*,*,*}=\sum_{\mu,\gamma,n}P_{\mu,\gamma,n}$.

Given any $\mu\in W$ and $\lambda\in G_0^*$, let $P_{\mu,\lambda,n}$
be the $k$-linear space spanned by all skew primitives $y\in H$
with generalized commutator $\lambda$ of level no more than $n$
and $\mu(y)=\mu$.
Let $P_{\mu,\lambda,*}=\sum_{n\geq 0}P_{\mu,\lambda,n}$,
$P_{*,G*,n}=\sum_{\mu,\lambda}P_{\mu,\lambda,n}$ and
$P_{*,G*,*}=\sum_{\mu,\lambda,n}P_{\mu,\lambda,n}$.

Lemma \ref{xxlem3.7}(a) says that $P_{*,*,*}$ contains
all skew primitive elements and Lemma \ref{xxlem3.7}(c)
says that $P_{\mu,\gamma,1}$ is a subspace of $P_{*,G*,*}$
when $\gamma\in k^\times$ and $\gamma$ is 1 or not a root of 
unity. 

\begin{lemma}
\label{xxlem3.8}
Retain the above notation. Suppose that $H$ does not contain 
a free subalgebra of rank 2. 
\begin{enumerate}
\item
If $P_{\mu,\gamma,1}\subset k(\mu-1)$, then $P_{\mu,\gamma, *}
\subset k(\mu-1)$.
\item
If $P_{\mu, \gamma,1}\not\subset k(\mu-1)$ and $\gamma$ is not
a root of unity, then $P_{\mu,\gamma_0,*}\subset k(\mu-1)$ for every
root of unity $\gamma_0$.
\item
If $\gamma$ is not a root of unity, then
$P_{\mu,\gamma,1}/k(\mu-1)$ has dimension at most 1.
\item
Suppose $\gamma_1$ and $\gamma_2$ are two distinct scalars
neither of which is a root of unity.
If $P_{\mu, \gamma_i,1}\not \subset k(\mu-1)$ for $i=1,2$,
then $\gamma_1^{N}\gamma_2^{M}=1$
for some positive integers $N,M$. Further there is no
$\gamma_3\in k^\times \setminus \{\gamma_1,
\gamma_2\}$ such that $P_{\mu,\gamma_3,1}
\not \subset k(\mu-1)$.
\item
Suppose that $\GKdim H<\infty$.
If $\gamma$ is not a root of unity, then $P_{\mu,\gamma,*}=
P_{\mu,\gamma,1}\subset P_{*,G*,*}$.
\end{enumerate}
\end{lemma}

\begin{proof} (a) This is clear by induction.

(b) First assume that $\gamma_0$ is not 1.
Pick $y_2\in P_{\mu, \gamma,1}\setminus C_0$. If
$P_{\mu,\gamma_0,*}\not\subset k(\mu-1)$ for a root of unity 
$\gamma_0$, by part (a), there is a $y_1\in P_{\mu,\gamma_0,1}
\setminus C_0$. Since $\gamma_0$ and $\gamma$ are not 1, 
after adding a suitable term $\alpha(\mu(y_i)-1)$ to $y_i$,
the hypothesis in Proposition \ref{xxprop3.3}(ii) holds.
Then we are in the situation of Proposition \ref{xxprop3.3}(b2,b3)
(where $d_1=d_2=1$). By Proposition \ref{xxprop3.3}, $H$ contains 
a free subalgebra of rank 2, a contradiction.

Second assume that $\gamma_0=1$. For the statement we only need to
consider the Hopf subalgebra generated by 
group-like and skew primitive elements. So we may assume
that $H$ is pointed. Passing to the associated graded Hopf
algebra of $H$ with respect to its coradical filtration,
one can assume the hypothesis in Proposition \ref{xxprop3.3}(ii) holds.
So the above proof works. 

(c) This follows from Proposition \ref{xxprop3.3}(b1).

(d) By the proof of Proposition \ref{xxprop3.3}, see \eqref{I3.3.1},
$$q_1^{d_1(M_1(M_1-1))+d_2(M_1M_2)}
q_2^{d_2(M_2(M_2-1))+d_1(M_1M_2)}=1$$
where $M_1$ and $M_2$ are nonnegative integers with $M_1+M_2\geq 2$.
Note that $\gamma_i=q_i$ and $d_1=d_2=1$ in this case.
Let $N=M_1(M_1-1)+M_1M_2$ and 
$M=M_2(M_2-1)+M_1M_2$.
Then $N$ and $M$ are non-negative and $N+M\geq 2$. Since $\gamma_1$
and $\gamma_2$ are not roots of unity, both $N$ and $M$ must be positive.

If such a $\gamma_3$ exists, by part (b) it is not a root of unity.
By the above assertion we have positive integers
$A,B,C,D$ such that $\gamma_1^A \gamma_3^B=1$ and
$\gamma_2^C \gamma_3^D=1$. Then
$$1=(\gamma_1^{N}\gamma_2^M)^{AC}=
\gamma_1^{ACN}\gamma_2^{CAM}=\gamma_3^{-BCN-DAM}$$
which contradicts the fact $\gamma_3$ is not a root of unity.

(e) If $P_{\mu,\gamma,*}\neq P_{\mu,\gamma,1}$, pick
a skew primitive $y_2\in P_{\mu,\gamma,2}$. This means that
$\mu^{-1}y_2 \mu-\gamma y_2=y_1\in P_{\mu,\gamma,1}
\setminus k(\mu-1)$. Consider the Hopf subalgebra $K$ generated 
as an algebra by $y_2, y_1, \mu, \mu^{-1}$. By possibly adding a 
term $\beta(\mu-1)$ to $y_1$ for some $\beta\in k$ and adding a 
term $\alpha(\mu-1)$ to $y_2$ for some $\alpha\in k$, we can assume 
that $y_1\mu =\gamma \mu y_1$ and $y_2\mu = \gamma \mu y_2+\mu y_1$.

It remains to show that $K$ contains a free subalgebra of rank 2.
Define a filtration $F_i$ of $K$ inductively as follows:

$F_0=C_0(K)=k[\mu^{\pm 1}]$,

$F_1= F_0+F_0 y_1=F_0 y_1 +F_0$,

$F_2= F_0+F_0 y_1+F_0 y_1^2+ F_0 y_2=
F_0+y_1 F_0 +y_1^2 F_0 + y_2 F_0$.

$F_n=\sum_{i=1}^{n-1} F_{i} F_{n-i}$ for all $n\geq 3$.

It is easy to check that $F$ is an Hopf algebra filtration and $\gr_F K$
is a Hopf algebra. In $\gr_F K$, $y_1,y_2$ are linearly
independent in $\gr_F K/k(\mu-1)$. Since $y_1,y_2\in P_{\mu, \gamma,1}(K)$,
by part (c), $\gr_F K$ contains a free subalgebra of rank 2. Therefore $K$
contains a free subalgebra of rank 2. This yields a contradiction. Therefore
the first assertion (namely, the first equation) follows. Finally, by 
Lemma \ref{xxlem3.7}(c), $P_{\mu,\gamma,1}\subset P_{*,G*,*}$.
\end{proof}

\begin{remark}
\label{xxrem3.9} Suppose $\GKdim H<\infty$. By Lemmas \ref{xxlem2.1}(b)
and \ref{xxlem3.7}(a),
$W\setminus W_{\sqrt{\;}}\subseteq W_{\times}$. In practice it often 
happens that $W\setminus W_{\sqrt{\;}}= W_{\times}$.
\end{remark}

Now we are ready to prove the Third Lower Bound Theorem.
Let $Y_*$ be the $k$-linear vector space spanned by all skew
primitive elements $y$ with commutator of finite level such that
$\gamma(y)\in \Gamma\setminus \Gamma_{\sqrt{\;}}$.

\begin{theorem}
\label{xxthm3.10}
Suppose $G_0$ is abelian. Then
$$\GKdim H\geq \GKdim C_0+\dim Y_*/(Y_*\cap C_0).$$
\end{theorem}

\begin{proof} Nothing needs to be proved if $\GKdim H=\infty$,
so we assume $\GKdim H<\infty$.

Let
$$Y_n=\sum_{\mu\in W_{\times},\gamma\in \Gamma\setminus
\Gamma_{\sqrt{\;}}}
P_{\mu,\gamma,n}\quad {\text{and}}\quad
Y_{Gn}= \sum_{\mu\in W_{\times}, \lambda\in G_0^*,
\gamma:=\lambda(\mu)\in \Gamma\setminus \Gamma_{\sqrt{\;}}}
P_{\mu,\lambda,n}$$
for all $n\geq 1$. 
Then $Y_*=\sum_n Y_n$. Let $Y_{G*}= \sum_n Y_{Gn}$. We prove the 
following claim by induction:

Claim A:
$$\GKdim H\geq \GKdim C_0+\dim Y_{Gn}/(Y_{Gn}\cap C_0)$$
for all $n\geq 1$. 
When $n=1$, let $\{y_i\}$ be a basis of $Y_{G1}/(Y_{G1}\cap C_0)$
such that each $y_i$ is in $P_{\mu_i,\lambda_i,1}$ for some
$\mu_i,\lambda_i$. Then we have
$$\mu_j^{-1}y_i\mu_j=\lambda_i(\mu_j) y_i+\tau_{ij}(\mu_i-1)$$
where $\lambda_i(\mu_i)= \gamma(y_i)$ is either 1 or not a root of 
unity and where $\tau_{ij}\in k$. By Theorem
\ref{xxthm1.5} (with $D=C_0$ and $A=kG_0$), 
$$\GKdim H\geq 
\GKdim C_0+\#\{y_i\}=\GKdim C_0+\dim Y_{G1}/(Y_{G1}\cap C_0),$$
which proves Claim A for $n=1$.

Now assume that Claim A holds for $n$. Without loss of 
generality we may assume that $H$ is generated as an algebra
by $C_0$ and all skew primitive elements of $H$. 
Define a filtration $F_i$ of $H$ as follows:

$F_0=C_0$,

$F_1= F_0+F_0 Y_{G1}=F_0+Y_{G1}F_0$,

$F_2= F_1^2+ F_0 Z$ where $Z$ is the $k$-linear
span of all skew primitive elements of $H$, and

$F_m=\sum_{i=1}^{m-1} F_{i} F_{m-i}$ for all $m\geq 3$.

\noindent
Then $\{F_m\}$ is a Hopf algebra filtration of $H$. Let $K$ be the 
associated graded Hopf algebra $\gr_F H$. Note that
if $y\in P_{\mu,\lambda,2}\subset Y_{G2}(H)$, then $y\in F_2$
and the associated element in $K$ is $\gr y\in F_2/F_1$ and
$g^{-1}(\gr y) g=\lambda(g) (\gr y)$ for all $g\in G_0$
(or, when trivially $y\in F_1$, we have $\gr y\in F_1/F_0$ or 
$\gr y\in F_0$). Thus $\gr y\in Y_{G1}(K)$ (which is easy to see
when $y\in F_1$). By induction one sees that
$$Y_{Gn}(K)\supseteq \{\gr y \; \mid \; y\in Y_{G(n+1)}(H)\}
=:\gr Y_{G(n+1)}(H)$$ 
for all $n$. Applying the induction hypothesis to $K$,
$$\GKdim K\geq \GKdim C_0+\dim Y_{Gn}(K)/(Y_{Gn}(K)\cap C_0).$$
By \cite[Lemma 6.5]{KL}, $\GKdim H\geq \GKdim K$.
It is clear that
$$\dim Y_{G(n+1)}(H)/(Y_{G(n+1)}(H)\cap
C_0)=\dim \gr Y_{G(n+1)}(H)/(\gr Y_{G(n+1)}(H)\cap C_0).$$
Therefore
$$
\begin{aligned}
\GKdim H &\geq \GKdim K\geq
\GKdim C_0+\dim Y_{Gn}(K)/(Y_{Gn}(K)\cap C_0)\\
&\geq \GKdim C_0+\dim \gr Y_{G(n+1)}(H)/(\gr Y_{G(n+1)}(H)\cap C_0)\\
&=\GKdim C_0+\dim Y_{G(n+1)}(H)/(Y_{G(n+1)}(H)\cap C_0)
\end{aligned}
$$
which finishes the induction step. Therefore we proved
Claim A.

When $n$ goes to infinity, we have
\begin{equation}
\label{I3.10.1}\tag{I3.10.1}
\GKdim H \geq
\GKdim C_0+\dim Y_{G*}(H)/(Y_{G*}(H)\cap C_0).
\end{equation}

Next we prove the following claim by induction:

Claim B:
$$\GKdim H\geq \GKdim C_0+\dim Y_{n}/(Y_{n}\cap C_0).$$
When $n=1$, this follows from \eqref{I3.10.1} since
$Y_1(H)\subset Y_{G*}(H)$ by Lemma \ref{xxlem3.7}(c).
Note that $Y_1$ is $G_0$-stable.
Using an argument similar to the proof of Claim A by passing to
the associated graded Hopf algebra $\gr_F H$ (and
replacing $Y_{Gn}$ by $Y_n$), one sees that Claim B holds.
When $n$ goes to infinity, we have
$$
\GKdim H \geq
\GKdim C_0+\dim Y_{*}(H)/(Y_{*}(H)\cap C_0).$$
\end{proof}

The proof of Theorem \ref{xxthm3.10} also shows the following.

\begin{corollary}
\label{xxcor3.11}
Suppose $G_0$ is abelian and $\GKdim H<\infty$. Let
$\mu\in W$ and suppose that $\gamma\in k^{\times}$ is not a root of unity.
Then
\begin{enumerate}
\item
$P_{\mu,1,*}$ is finite dimensional.
\item
$P_{\mu,1,*}$ and $P_{\mu,\gamma,*}$ are subspaces
of $P_{*,G*,*}$.
\item
$Y_{*}=Y_{G*}$.
\end{enumerate}
\end{corollary}

By Lemma \ref{xxlem3.8}(c,e), $P_{\mu,\gamma,*}$ is finite dimensional.

\begin{proof}[Proof of Corollary \ref{xxcor3.11}] 
(a) By Theorem \ref{xxthm3.10}, $Y_*/(Y_*\cap C_0)$
is finite dimensional. Since $P_{\mu,1,*}\subset Y_*$ and
$P_{\mu,1,*}\cap C_0= k(\mu-1)$, we have
$$\dim P_{\mu,1,*}/k(\mu-1)\leq \dim Y_*/(Y_*\cap C_0)<\infty$$
which implies that $P_{\mu,1,*}$ is finite dimensional.

(b) By Lemma \ref{xxlem3.8}(e) $P_{\mu,\gamma,*}$ is a subspace
of $P_{*,G*,*}$.

By part (a) $P_{\mu,1,*}$ is finite dimensional. It is clear
that $P_{\mu,1,*}$ is $G_0$-stable. By Lemma \ref{xxlem3.7}(b,c)
$P_{\mu,1,*}$ is a subspace of $P_{*,G*,*}$.

(c) As a consequence of part (b), $P_{\mu, \lambda, *}=
P_{\mu,\lambda(\mu),*}$ when $\lambda(\mu)$ is either 1 or not a
root of unity. The assertion follows.
\end{proof}

Theorem \ref{xxthm0.4} is an immediate consequence of Theorem
\ref{xxthm3.10}.

\begin{proof}[Proof of Theorem \ref{xxthm0.4}]
Without loss of generality we assume that $\GKdim H<\infty$.
First we claim that $Y_*\cap (Y_{\sqrt{\;}}+C_0)\subset C_0$.
Suppose $y=z+c$ is in $Y_*\cap (Y_{\sqrt{\;}}+C_0)$ where 
$y\in Y_*\setminus C_0$ and $z\in Y_{\sqrt{\;}}$ and $c\in C_0$. 
It is easily reduced to the case when $\mu(y)=\mu(z)=g$ and 
$c=\alpha(g-1)$ for some $\alpha\in k$. Then $y\in Y_{\sqrt{\;}}$,
a contradiction. Therefore the claim holds. 

By Lemma \ref{xxlem3.7}(a), every skew primitive
element is a linear combination of skew primitives with commutator
of finite level. This says that $Z=Y_{*}+Y_{\sqrt{\;}}+C_0$. Since
$Y_*\cap (Y_{\sqrt{\;}}+C_0)\subset C_0$, we have
$$Z/(C_0+Y_{\sqrt{\;}})\cong Y_*/Y_*\cap (C_0+Y_{\sqrt{\;}})=
Y_*/(Y_*\cap C_0).$$
The assertion follows by Theorem \ref{xxthm3.10}.
\end{proof}

Another way of proving Theorem \ref{xxthm0.3} (with a slightly 
stronger hypothesis that $\langle W_{\times}\rangle$ is abelian) 
is using Theorem \ref{xxthm0.4} and the following lemma, which 
is due to an anonymous referee.

\begin{lemma}
\label{xxlem3.12} Suppose $\GKdim H<\infty$. Then
$$\dim Z/(C_0+Y_{\sqrt{\;}})\geq \# (W\setminus W_{\sqrt{\;}}).$$
\end{lemma}

\begin{proof} For any $w\leq \# (W\setminus W_{\sqrt{\;}})$,  
take distinct elements $x_1,\cdots,x_w \in W\setminus W_{\sqrt{\;}}$, 
we need to prove that $\dim Z/(C_0+ Y_{\sqrt{\;}})\geq w$.

For each $i$, pick a skew primitive element $z_i\in H\setminus C_0$ 
such that $x_i=\mu(z_i)$. Lemma \ref{xxlem2.5} shows how to get a 
suitable $z_i$ such that $\gamma(z_i)$ is defined. Since $\mu(z_i)=
x_i\not\in W_{\sqrt{\;}}$, Lemma \ref{xxlem2.1}(b) says that 
$\gamma(z_i)$ is either 1 or not a root of unity. 

It suffices to show that $z_1,\cdots,z_w$ are linearly independent 
in $Z/(C_0+Y_{\sqrt{\;}})$, so it is enough to show that 
$z_1,\cdots,z_w,y$ are linearly independent in $Z/C_0$ for any 
$y\in Y_{\sqrt{\;}}\setminus C_0$. We can arrange $y = y_1+\cdots+
y_v$ for some skew primitive elements $y_j\not\in C_0$, 
where $y_j$ has a commutator $\gamma(y_j)$ of finite level which is 
a nontrivial root of unity, and the pairs $(\mu(y_j),\gamma(y_j))$ 
are distinct. The pairs $(\mu(z_i),\gamma(z_i))$ are already distinct, 
and $(\mu(z_i),\gamma(z_i))\neq (\mu(y_j),\gamma(y_j))$ for all 
$i,j$ because $\gamma(z_i)$ is either 1 or not a root of unity.

An improved version of Lemma \ref{xxlem2.3} for skew primitive elements 
with commutators of finite level says that $z_1,\cdots, z_w,y_1, 
\cdots,y_v$ are linearly independent in $Z/C_0$. Consequently,
$z_1,\cdots,z_w, y$ are linearly independent in $Z/C_0$ as desired.
\end{proof}

Finally we end with a simple example in which $Z=Y_{\sqrt{\;}}+C_0$.

\begin{example}
\label{xxex3.13}
Let $H$ be generated as an algebra by $x$ and $\{y_i\}_{i=1}^{\infty}$ 
subject to the relations
$$\begin{aligned}
x^2&=1,\\
y_i^2&=0,\\
y_iy_j+y_jy_i&=0,\\
xy_i+y_ix&=0
\end{aligned}
$$
for all $i,j\in {\mathbb N}$. The coalgebra structure and the antipode
of $H$ are determined by
$$\begin{aligned}
\Delta(x)=x\otimes x, \qquad & \Delta(y_i)=y_i\otimes 1+x\otimes y_i,\\
\epsilon(x)=1, \qquad & \epsilon(y_i)=0,\\
S(x)=x,  \qquad & S(y_i)=-x y_i=y_i x
\end{aligned}
$$
for all $i$. It is easy to check the following
\begin{enumerate}
\item
$\GKdim H=0$,
\item
$\Omega=\{(x,-1)\}=\Omega_{\sqrt{\;}}$,
\item
$Z= Y_{\sqrt{\;}}+C_0$ and $\dim Y_{\sqrt{\;}}=\infty$,
\item
$P_{x,-1,*}=P_{*,-1,1}=Y_{\sqrt{\;}}$.
\end{enumerate}
\end{example}


\subsection*{Acknowledgments}
The authors thank Ken Goodearl and Ken Brown for their valuable 
comments and thank the referee for his/her careful reading, many 
comments, corrections, suggestions and Lemma \ref{xxlem3.12}.
A part of this research was done when J.J. Zhang was visiting Fudan
University in the Fall of 2009 and the Spring of 2010.
D.-G. Wang was supported by the National Natural Science Foundation 
of China (No. 10671016  and  11171183) and the Shandong Provincial 
Natural Science Foundation of China (No. ZR2011AM013).
J.J. Zhang and G. Zhuang were supported by the US National
Science Foundation (NSF grant No. DMS 0855743). 

\bigskip

\providecommand{\bysame}{\leavevmode\hbox to3em{\hrulefill}\thinspace}
\providecommand{\MR}{\relax\ifhmode\unskip\space\fi MR }
\providecommand{\MRhref}[2]{%

\href{http://www.ams.org/mathscinet-getitem?mr=#1}{#2} }
\providecommand{\href}[2]{#2}

\end{document}